%% file: article.tex
\documentclass[onefignum,onetabnum]{siamart250211}

\input{shared}

\ifpdf
\hypersetup{
  pdftitle={Enforcing boundary conditions},
  pdfauthor={N. Göschel, S. Götschel and D. Ruprecht}
}
\fi

\begin{document}

\maketitle

\begin{abstract}
Machine-learning based methods like physics-informed neural networks and physics-informed neural operators are becoming increasingly adept at solving even complex systems of partial differential equations.
Boundary conditions can be enforced either weakly by penalizing deviations in the loss function or strongly by training a solution structure that inherently matches the prescribed values and derivatives.
The former approach is easy to implement but the latter can provide benefits with respect to accuracy and training times.
However, previous approaches to strongly enforcing Neumann or Robin boundary conditions require a domain with a fully $C^1$ boundary and, as we demonstrate, can lead to instability if those boundary conditions are posed on a segment of the boundary that is piecewise $C^1$ but only $C^0$ globally.
We introduce a generalization of the approach by Sukumar \& Srivastava (doi: 10.1016/j.cma.2021.114333),  and a new approach based on orthogonal projections that overcome this limitation.
The performance of these new techniques is compared against weakly and semi-weakly enforced boundary conditions for the scalar Darcy flow equation and the stationary Navier-Stokes equations.
\end{abstract}

\begin{keywords}
physics-informed neural operators, physics-informed neural networks, boundary conditions, approximate distance function, hard constraints
\end{keywords}


\section{Introduction}\label{sec:intro}
\input{introduction.tex}

\section{Background}\label{sec:related_work}
\input{related_work.tex}

\section{Methodology}\label{sec:methodology}
\input{methodology.tex}

\section{Numerical results}\label{sec:examples}
\input{results.tex}

\section{Conclusion}\label{sec:conclusion}
\input{conclusion.tex}


\appendix
\input{appendix.tex}

\bibliographystyle{siamplain}
\bibliography{references}
\end{document}

%% file: shared.tex
\usepackage{lipsum}
\usepackage{amsfonts}
\usepackage{graphicx}
\usepackage{epstopdf}

\ifpdf
  \DeclareGraphicsExtensions{.eps,.pdf,.png,.jpg}
\else
  \DeclareGraphicsExtensions{.eps}
\fi

\usepackage{subcaption}
\usepackage{tikz}
\usepackage{booktabs}
\usepackage{algorithm}
\usepackage{algpseudocode}
\usepackage{mathtools}

\DeclareMathOperator{\Span}{span}

\newsiamremark{remark}{Remark}
\newsiamremark{hypothesis}{Hypothesis}
\crefname{hypothesis}{Hypothesis}{Hypotheses}
\newsiamthm{claim}{Claim}
\newsiamremark{fact}{Fact}
\crefname{fact}{Fact}{Facts}

\headers{Enforcing boundary conditions}{N. Göschel, S. Götschel and D. Ruprecht}

\title{Enforcing boundary conditions for physics-informed neural operators
\thanks{This work was funded by the Deutsche Forschungsgemeinschaft (DFG) within CRC 1615: SMART Reactors for Future Process Engineering -- 503850735.}}

\author{Niklas Göschel\thanks{Chair Computational Mathematics, Institute of Mathematics, Hamburg University of Technology, 21073 Hamburg, Germany
  (\email{niklas.goeschel@tuhh.de}, \email{sebastian.goetschel@tuhh.de}, \email{ruprecht@tuhh.de}).}
\and Sebastian Götschel\footnotemark[2]
\and Daniel Ruprecht\footnotemark[2]
}

\usepackage{amsopn}


%% file: introduction.tex
Various machine learning-based techniques have been applied successfully to solve systems of partial differential equations (PDEs).
Most prominently, physics-informed neural networds~\cite{Raissi19} and many variants and extensions~\cite{Raissi24} learn the solution to a PDE. 
They use (often dense) neural networks combined with loss functions that include the PDE residuals to promote physically meaningful solutions and reduce the amount of training data required.
Neural operators~\cite{li20a, Lu21a} were developed to learn solution operators of PDEs.
Based on the Fourier neural operator (FNO) by~\cite{li20b}, a physics-informed neural operator (PINO) was proposed by~\cite{li21}.
Similar to PINNs it was designed to approximate the solution operator of a parametric  partial differential equation by minimizing  a residual given by the differential equation instead of training solely on labeled training data.

In physics-informed machine learning, boundary conditions can be enforced in two ways. 
One is to weakly enforce them by adding a residual term that punishes but does not prohibit differences to the prescribed values.
An extension are penalty methods to treat boundary conditions as hard constraints in the optimization~\cite{Lu21b}.
However,~\cite{Toscano25,Zeinhofer24} show that these approaches weaken the decay of the generalization error.
The alternative is to strongly enforce boundary conditions by constructing the solution in a way that it exactly satisfies the boundary conditions. 
While this is straightforward for Dirichlet boundary conditions~\cite{Berrone23,Toscano25}, Neumann- or Robin boundary conditions are harder to treat.
Techniques to do this include Fourier feature embeddings~\cite{Straub25} or solution structures based on trial solutions using distance functions~\cite{Manavi24,McFall09}.
Based on the latter idea,~\cite{Sukumar21} introduce a flexible method to strongly enforce boundary conditions and observe that it can improve accuracy for PINNs.

\paragraph{Contributions and structure of the paper}
We discuss prescribing boundary conditions in the setting of physics-informed neural operators and -networkd in some detail in\S~\ref{sec:related_work}, and summarize the approach by Sukumar and Srivastava~\cite{Sukumar21}. We will show that their approach is suitable for a certain class of boundary conditions but fails when Neumann conditions are prescribed on a boundary that is not $C^1$. 
We propose two novel approaches to prescribe Robin or Neumann boundary conditions on boundaries that are piecewise $C^1$ but only $C^0$ globally in \S~\ref{sec:methodology}.
First we describe a generalization of the method by~\cite{Sukumar21} called \emph{generalized local solution structures} or GLSS.
The second approach is based on \emph{orthogonal projections} and we refer to it as OP.
While it requires certain assumptions on the shape of the boundary, it has fewer unknown functions that need to be learned by the network.
GLSS and OP are described in \S~\ref{subsec:pde_scalar} for scalar PDEs and in \S~\ref{subsec:pde_systems} for systems of PDEs. We provide algorithms for the generation of these solution structures in \S~\ref{sec:algorithms}. Numerical results are discussed in \S~\ref{sec:examples}.
\S~\ref{subsec:darcyy} compares GLSS, OP as well as weakly and semi-weakly boundary conditions for a scalar PDE, the Darcy flow equation.
Finally, \S~\ref{subsec:nse} compares their performance for a standard benchmark from computational fluid dynamics by~\cite{Turek96} that requires solving the stationary Navier-Stokes equations to model flow around a cylinder.

The Python code for the implementation and all examples can be found at~\url{https://zenodo.org/records/17453006}.
%

%% file: related_work.tex
In this section we discuss approaches to prescibe boundary conditions in physics-informed machine learning, revisit the approach by Sukumar and Srivastava~\cite{Sukumar21}, and illustrate the arising issues.

Here and in the remainder of the article, let $\Omega \subset \mathbb{R}^2$ be a computational domain, $\mathcal{U}, \mathcal{V}$ Banach spaces, and let $\mathcal{A} \subset \mathcal{V}$ be a set of parameter functions.
For a $\boldsymbol{a} \in \mathcal{A}$ we want to find the solution $\boldsymbol{u} \in\mathcal{U}$ to the boundary value problem
\begin{subequations}
	\label{eq:pde_plus_bc}
	\begin{align}
		\forall \boldsymbol{x} & \in \Omega : & \mathcal{P} (\boldsymbol{u} (\boldsymbol{x}),\boldsymbol{a} (\boldsymbol{x})) & = 0, \label{eq:pde} \\
		\forall \boldsymbol{x} & \in \partial \Omega : & \mathcal{B} (\boldsymbol{u} (\boldsymbol{x}),\boldsymbol{a} (\boldsymbol{x})) & = 0, \label{eq:bc}
	\end{align}
\end{subequations}
where $\mathcal{P}$ is a differential operator and $\mathcal{B}$ is a boundary condition operator.

\subsection{Physics-informed neural networks and physics-informed neural operators}
For a physics-informed neural operator (PINO), the aim is to learn the solution operator $\mathcal{G}_{\boldsymbol{\theta}} : \mathcal{A} \rightarrow \mathcal{U}$ with $\mathcal{G}_{\boldsymbol{\theta}}(\boldsymbol{a} ) = \boldsymbol{u}$ using a set of training parameters $\mathcal{A}_{\text{train}} \subset \mathcal{A}$. 
We denote the trainable parameters of the neural operator as $\boldsymbol{\theta}$.
Note that~\cite{Sukumar21} present their approach in the context of physics-informed neural network (PINNs) whereas we consider PINOs.
However, in the notation above, a PINN learns $u_\theta(\boldsymbol{a^\star}) = G_{\theta}(\boldsymbol{a^\star})$ for a \emph{fixed} parameter $\boldsymbol{a^\star}$.
That is, it learns one specific instance $G_{\theta}(\boldsymbol{a^\star})$ solving the boundary problem~\eqref{eq:pde}.   
By contrast, a PINO trains on a large set of parameters to learn to mapping $\boldsymbol{a} \mapsto G_{\theta}(\boldsymbol{a})$. 
The learned solution operator can be further refined to compute the solution for a specific $\boldsymbol{a^\star}$ through continued training only on this parameter (\emph{finetuning}).

In the following, we use the FNO-PINO architecture to learn not a mapping to the PDE solution directly, but to the unknown functions in the \emph{solution structure} (Def.~\ref{def:solution_structure}). 
For simplicity, we utilize the FNO framework even though it requires a rectangular domain with a uniform mesh to use the Fast Fourier Transform (FFT).
To work on more complex geometries we follow the approach by~\cite{lu21} and choose the minimum bounding box of the underlying domain as computational domain. 
Alternatively, for more efficient approaches one could utilize the geo-FNO framework proposed by~\cite{li22} or geometry-informed neural operators~\cite{Li23}.

With slight abuse of notation we will refer to the output of the neural networks still as $\mathcal{G}_{\boldsymbol{\theta}}(\boldsymbol{a})$. 
Our approaches to enforce boundary conditions can be used for either PINN or PINO. 
To illustrate applicability to PINNs, in addition to the regular training and finetuning of the PINO, we consider \emph{PINN-like training}, i.\,e. learning the solution structure for a specific parameter $\boldsymbol{a^\star}$ without training the solution operator first.
This approach is essentially a PINN that uses the FNO architecture instead of dense layers.
As we focus on boundary conditions, we do not consider data mismatch terms in the loss, but only physics losses.
There are three ways to ensure that the trained solution $\boldsymbol{u}$ satisfies the boundary condition~\eqref{eq:pde_plus_bc}.
\paragraph{Weak boundary conditions}
Here, the neural operator outputs the solution directly, that is $\boldsymbol{u} = \mathcal{G}_{\boldsymbol{\theta}}(\boldsymbol{a})$, and satisfying the boundary conditions has to be learned during training.
This corresponds to solving the minimization problem
\begin{equation}
	\min_{\boldsymbol{\theta}} \sum_{\boldsymbol{a} \in \mathcal{A}_{\text{train}}}
	\left(
	\int_{\Omega} \mathcal{P} ((\mathcal{G}_{\boldsymbol{\theta}} (\boldsymbol{a})) (\boldsymbol{x}),\boldsymbol{a} (\boldsymbol{x}))^2 d\boldsymbol{x} 
	+ \int_{\partial \Omega} \mathcal{B} ((\mathcal{G}_{\boldsymbol{\theta}} (\boldsymbol{a})) (\boldsymbol{x}),\boldsymbol{a} (\boldsymbol{x}))^2 dS(\boldsymbol{x})
	\right).
\end{equation}
\paragraph{Exact boundary conditions} 

The idea by~\cite{Sukumar21} and~\cite{Rvachev95} is to train suitable functions $\Psi_i$ such that $\tilde{\boldsymbol{u}}(\Psi_1,\dots,\Psi_I)$ minimizes the PDE-loss from~\eqref{eq:pde} and, by construction, satisfies the boundary condition exactly.

\begin{definition}[Solution structure]\label{def:solution_structure}
	We call $\tilde{\boldsymbol{u}} : \{\bar{\Omega} \rightarrow \mathbb{R}^{I}\} \rightarrow \{\bar{\Omega} \rightarrow \mathbb{R}^{n}\}$ a solution structure, if 
	$\tilde{\boldsymbol{u}}(\Psi_1,\dots,\Psi_I)$ satisfies~\eqref{eq:bc} for any differentiable functions $\Psi_1, \dots, \Psi_I : \bar{\Omega} \rightarrow \mathbb{R}$.
\end{definition}

For exact boundary conditions, such a solution structure is constructed, and functions $\Psi_i$ trained such that the parameters satisfy
\begin{equation}
	\min_{\boldsymbol{\theta}} \sum_{\boldsymbol{a} \in \mathcal{A}_{\text{train}}}
	\int_{\Omega} \mathcal{P} (\tilde{\boldsymbol{u}}(\mathcal{G}_{\boldsymbol{\theta}} (\boldsymbol{a})) (\boldsymbol{x}),\boldsymbol{a} (\boldsymbol{x}))^2 d\boldsymbol{x}.
\end{equation}
for the resulting $\tilde{\boldsymbol{u}}(\Psi_1,\dots,\Psi_I)$. 
\paragraph{Semi-weak / semi-exact boundary conditions} 
Lastly we consider a solution structure that satisfies some boundary conditions but not all.
Let $\tilde{\mathcal{B}}$ represent the boundary conditions that $\tilde{\boldsymbol{u}}(\Psi_1,\dots,\Psi_I)$ does not autoamtically satisfy.
The minimization problem solved in training is then posed as a combination of the two previous ones 
\begin{equation}
	\min_{\boldsymbol{\theta}} \sum_{\boldsymbol{a} \in \mathcal{A}_{\text{train}}}
	\left(
	\int_{\Omega} \mathcal{P} (\tilde{\boldsymbol{u}}(\mathcal{G}_{\boldsymbol{\theta}} (\boldsymbol{a})) (\boldsymbol{x}),\boldsymbol{a} (\boldsymbol{x}))^2 d\boldsymbol{x}
	+ \int_{\partial \Omega} \tilde{\mathcal{B}} ((\mathcal{G}_{\boldsymbol{\theta}} (\boldsymbol{a})) (\boldsymbol{x}),\boldsymbol{a} (\boldsymbol{x}))^2 dS(\boldsymbol{x})
	\right).
\end{equation}
We will later use this approach to enforce Dirichlet boundary conditions exactly and Robin conditions weakly as a comparison.

\subsection{Approximate distance functions and solution structures}

In their foundational work, Sukumar and Srivastava~\cite{Sukumar21} utilize the theory of R-functions~\cite{Rvachev95} and approximate distance functions to train solutions to boundary value problems using PINNs.
An approximate distance function satisfies two properties: it has to be a \emph{distance function}, as well as \emph{normalized} with respect to the boundary. More precisely, we have the following definitions:
\begin{definition}[Distance function]
  \label{def_distance_function}
  Let $\Gamma \subset \partial \Omega$ be a boundary section.
  We call a function $\phi : \bar{\Omega} \rightarrow \mathbb{R}$ a \emph{distance function to $\Gamma$} if
  it is zero on $\Gamma$ and positive in $\bar{\Omega} \setminus \Gamma$.
\end{definition}
The other definition is according to~\cite{Rvachev95}.
\begin{definition}[Normalized function]
  \label{def_normalized_function}
  Let $\Gamma \subset \partial \Omega$ be a boundary section.
  We call a function $\bar{\phi} : \bar{\Omega} \rightarrow \mathbb{R}$ 
  a \emph{normalized function with respect to $\Gamma$} if it satisfies $\bar{\phi} \equiv 0$ and $ \frac{\partial \bar{\phi}}{\partial \boldsymbol{\nu}} \equiv 1$ on $\Gamma$.
\end{definition} 
Here, $\boldsymbol{\nu}$ denotes the inward pointing normal vector on $\partial \Omega$.
\begin{definition}[Approximate distance function]
  Let $\Gamma \subset \partial \Omega$ be a boundary section.
  We call a function $\phi : \bar{\Omega} \rightarrow \mathbb{R}$ an \emph{approximate distance function to $\Gamma$}, if $\phi$ is both a distance function to $\Gamma$ in the sense of Definition~\ref{def_distance_function} and normalized with respect to $\Gamma$ in the sense of Definition~\ref{def_normalized_function}.
\end{definition}

With these, solution structures can be constructed. 
We illustrate Sukumar and Srivastava's approach~\cite{Sukumar21} first for a Dirichlet boundary condition
\begin{equation}
	\label{eq:dirichlet-bc}
	\forall \boldsymbol{x} \in \partial \Omega : \quad u(\boldsymbol{x}) = g(\boldsymbol{x}),
\end{equation}
where $\Omega$ is the domain on which the corresponding differential equation is posed.
Let $\phi$ be the approximate distance function to $\partial \Omega$. 
Then, the solution is given by
\begin{equation}
	\label{solution_structure_one_Dirichlet_segment}
	u(\boldsymbol{x}) = g(\boldsymbol{x}) + \phi (\boldsymbol{x}) \Psi (\boldsymbol{x}).
\end{equation}
The function $\Psi$ is learned by the neural network and, because $\phi(\boldsymbol{x}) = 0$ on $\partial \Omega$, $ u(\boldsymbol{\boldsymbol{x}})$ satisfies the boundary condition~\eqref{eq:dirichlet-bc} by construction.

For a Robin boundary condition 
\begin{equation}
	\label{eq:robin-bc}
	\forall \boldsymbol{x} \in \partial \Omega : \quad \frac{\partial u(\boldsymbol{x})}{\partial \boldsymbol{n}} + c(\boldsymbol{x}) u(\boldsymbol{x}) = h(\boldsymbol{x}),
\end{equation}
the solution becomes
\begin{equation}
	\label{solution_structure_one_Robin_segment}
	u (\boldsymbol{x}) = 
	\underbrace{\Psi_{1} (\boldsymbol{x}) - \phi(\boldsymbol{x}) \nabla \phi(\boldsymbol{x}) \cdot \nabla \Psi_{1} (\boldsymbol{x})}_{\text{boundary value}}
	+ \underbrace{\phi(\boldsymbol{x}) \left( c(\boldsymbol{x}) \Psi_{1} (\boldsymbol{x}) - h(\boldsymbol{x}) \right)}_{\text{boundary derivative}}
	+ \underbrace{\phi (\boldsymbol{x})^2 \Psi_2 (\boldsymbol{x})}_{\text{remainder term}}
\end{equation}
It now depends on two functions $\Psi_1$ and $\Psi_2$ that the network needs to learn.
The first term determines the value on the boundary but has a normal derivative of zero whereas the second term is zero on the boundary but determines the normal derivative on the boundary.
If $\phi$ is normalized with respect to $\partial \Gamma$, the Robin condition~\eqref{eq:robin-bc} is satisfied. 
Because of $\phi^2$, the remainder term influences neither the value nor the derivative on at the boundary.

This approach can be extended to the case where both Dirichlet and Robin boundary conditions are prescribed on different boundary segments $\Gamma_1, \dots, \Gamma_M$.
We assume that each $\Gamma_i$ is $C^1$, but the overall boundary $\partial \Omega = \bigcup_{i=1}^{M} \Gamma_i$ need only be in $C^0$. \footnote{Note that this could be relaxed to $\partial \Omega \supset \bigcup_{i=1}^{M} \Gamma_i$ which would, for example, allow to treat initial-boundary value problems where the now boundary condition cannot be prescribed at final time. However, this case is not considered here.} 
The boundary conditions are
\begin{subequations}
	\label{eq:bc-appendix}
	\begin{align}
		\forall i & \in I_D : \forall \boldsymbol{x} \in \Gamma_i : & 
		u (\boldsymbol{x}) & = g_i (\boldsymbol{x}),  \label{eq:dirichlet_bc} \\
		\forall i & \in I_R : \forall \boldsymbol{x} \in \Gamma_i : & 
		\frac{\partial u (\boldsymbol{x})}{\partial \boldsymbol{n}} + c_i(\boldsymbol{x}) u (\boldsymbol{x}) & = h_i (\boldsymbol{x}). \label{eq:robin_bc}
	\end{align}
\end{subequations}
We consider Neumann boundary conditions as a special case of Robin conditions with $c_i \equiv 0$.

In \cite{Sukumar21}'s approach, the local solution structures $u_i$ are chosen as
\begin{equation}
	\label{local_solution_structure_Sukumar_ansatz}
	u_i (\boldsymbol{x}) = \begin{cases}
		g_i(\boldsymbol{x}), & i \in I_D \\
		\Psi_{i} (\boldsymbol{x}) - \phi_i(\boldsymbol{x}) \nabla \phi_i(\boldsymbol{x}) \cdot \nabla \Psi_{i} (\boldsymbol{x}) + \phi_i (\boldsymbol{x}) \left( c_i(\boldsymbol{x}) \Psi_i (\boldsymbol{x}) - h_i (\boldsymbol{x}) \right), & i \in I_R  
	\end{cases}.
\end{equation}
\begin{proposition}
	We have $u (\boldsymbol{x})  = u_i (\boldsymbol{x})$ for any $1 \leq i \leq M$ and any $\boldsymbol{x} \in \Gamma_i$.
\end{proposition}
Therefore, for $i \in I_D$ the Dirichlet boundary condition~\eqref{eq:dirichlet_bc} is satisfied by construction.
\begin{definition}[Distance Function with vanishing gradient]
	\label{def:vanishing_gradients}
	Let $\phi$ be a distance function to $\Gamma$. We say that $\phi$ has a \emph{vanishing gradient} if $\nabla \phi(\boldsymbol{x}) = 0$ for $\boldsymbol{x} \in \Gamma$.
\end{definition}
If the approximate distance functions corresponding the the Robin boundary segments have vanishing gradient, the boundary condition is satisfied by construction.
However, as pointed out by \cite{Gladstone22}, the non-differentiability of the approximate distance function becomes an issue for Neumann or Robin boundary value on boundaries that are not globally $C^1$.

\subsection{The issue with non-$C^1$ approximate distance functions}\label{subsec:issue}

As pointed out already by~\cite{Sukumar21}, suitable approximate distance functions should be $C^1$ in order to allow for Robin- or Neumann boundary conditions, as derivatives are required in the solution structure. Moreover, for collocation-based methods like PINNs, even higher regularity is needed, as otherwise the Laplacian of the distance function is unbounded at points where the distance function is only $C^0$, which causes issue when solving second-order differential equations.

To illustrate the issue that can arise if $\partial \Omega \notin C^1$, consider a simple Poisson problem with homogeneous Neumann boundary condition
\begin{equation}
  \forall (x,y)  \in (0,1)^2 : \quad
  - \Delta u (x,y)  = 2 \pi^2 \cos(\pi x) \cos(\pi y).
\end{equation}
Figure~\ref{fig:test_neumann} shows resulting solution (middle) as well as the analytical solution (left) and the solution using our generalized approach presented in this paper (right) for comparison.
The issue in the middle figure stems from the emergence of instabilities in the corner of the Laplacian of the approximate distance function, cf.~\cite[Figure 27]{Sukumar21}.

\begin{figure}[th]
	\centering
  \begin{subfigure}[t]{.32\textwidth}
    \includegraphics[width=\textwidth]{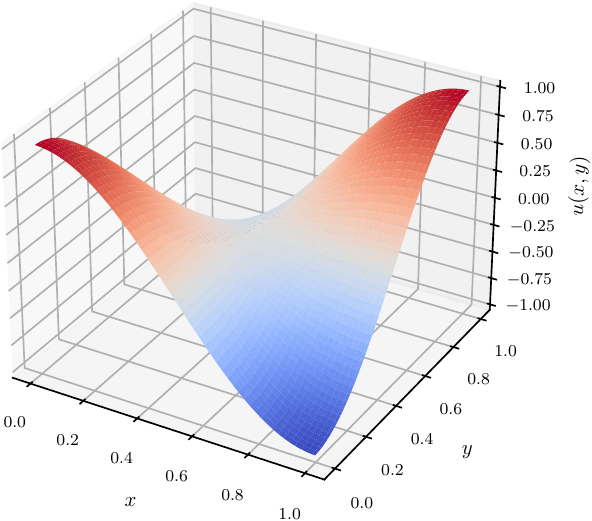}
    \caption{Analytical}
  \end{subfigure}
  \begin{subfigure}[t]{.32\textwidth}
    \includegraphics[width=\textwidth]{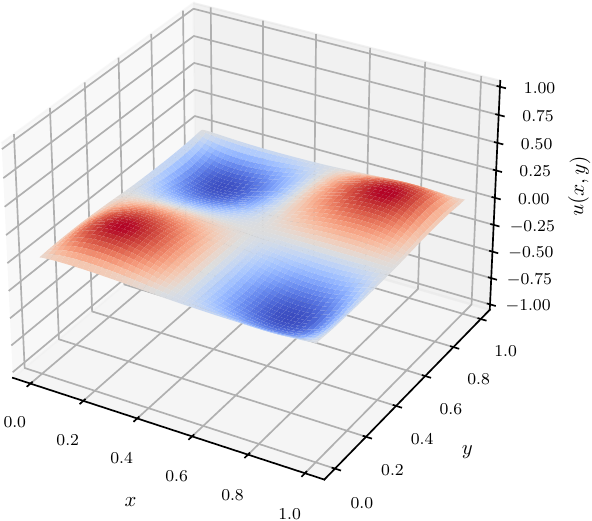}
    \caption{\cite{Sukumar21}}
  \end{subfigure} 
  \begin{subfigure}[t]{.32\textwidth}
    \includegraphics[width=\textwidth]{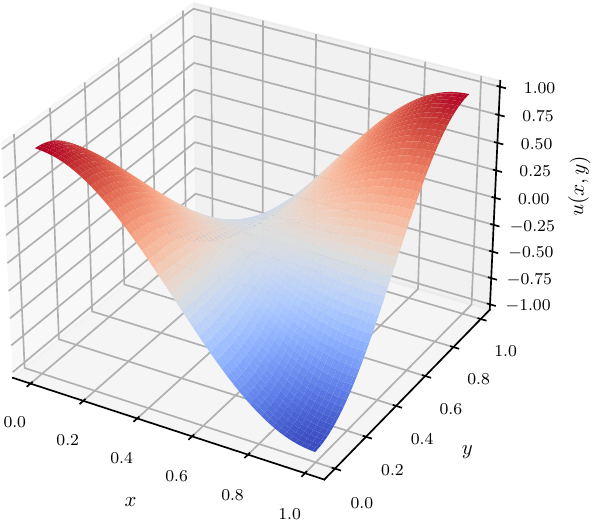}
    \caption{GLSS (our approach)}
  \end{subfigure}
	\caption{Analytical solution (left) and solution with~\cite{Sukumar21}'s (middle) and our generalized approach (right) to strongly enforcing boundary conditions.}
	\label{fig:test_neumann}
\end{figure}

While the exact distance function $d(\boldsymbol{x}) := \min_{\tilde{\boldsymbol{x}} \in \partial \Omega} ||\boldsymbol{x} - \tilde{\boldsymbol{x}}||$ is an approximate distance function to $\partial \Omega$, it is in general not $C^1$.
The construction of approximate distance functions to boundaries that consist of piecewise linear segments which are only $C^0$ globally is discussed in \cite{Sukumar21}.
However, their approximate distance functions are $C^1$ in the interior but not at the joining points of the segments.
This is not problem for the Dirichlet conditions or Neumann boundary conditions on the annulus they consider, where the boundary is globally $C^1$.

For many relevant cases, it is easy to show that $C^1$-regular approximate distance functions cannot exist. Consider, for sake of a simple example, the L-shaped domain~\ref{fig:L_shape}. Assume we could find a normalized distance function $\phi$ for $\Gamma_4$ that is $C^1$ everywhere, including in the point $D$. In this case, we could expand $\phi$ as $\phi(D + t(-1,1)^T) = \phi(D) + t(-1,1) \nabla \phi(D) + o( || t(-1,1)^T || )$, for $t\in \mathbb{R}$. We have  $\phi (D) = 0$ since $\phi$ is a distance function, and  $\nabla \phi (D) = (0,-1)^T$, as $\phi$ is normalized with respect to $\Gamma_4$. We thus have $\phi (D + t(-1,1)^T) = 0 + t(-1,1) (0,-1)^T + \mathcal{o} ( || t(-1,1)^T || ) = -t + \mathcal{o} ( | t | )$. For $t$ small enough this yields $| \phi(D + t(-1,1)^T) + t | < |t|$. However, if $t$ is small enough such that $D + t(-1,1)^T \in \Omega$ we have $\phi (D + t(-1,1)^T) > 0$. Therefore, $| \phi_4 (D + t(-1,1)^T) + t | >|t|$, which is a contradiction.

%% file: methodology.tex
To overcome the limitations discussed in \S~\ref{subsec:issue}, we relax the requirement of using approximate distance functions in the solution structures. We describe our approach first for scalar partial differential equations, and then generalizing it to systems of PDEs.
\subsection{Exact boundary conditions for scalar differential equations}\label{subsec:pde_scalar}

Starting from~\cite{Sukumar21}, the solution consists of two parts, the transfinite interpolant by~\cite{Rvachev01} for the boundary and a remainder term in the domain 
\begin{align}
\label{eq:generic_solution_structure_scalar}
& u (\boldsymbol{x}) = \underbrace{\sum_{i=1}^{M} w_i (\boldsymbol{x}) u_i (\boldsymbol{x})}_{\text{boundary}}
\enspace + \enspace \underbrace{\Psi (\boldsymbol{x}) \prod_{i=1}^{M} \phi_i (\boldsymbol{x})^{\mu_i}}_{\text{domain}}, \\
\label{generic_solution_structure_scalar_weight_functions}
\forall i \in \{1,\dots,M\} : \quad &
w_i (\boldsymbol{x}) = \frac{ \prod_{j=1, j \neq i}^M \phi_j (\boldsymbol{x})^{\mu_j} }{ \sum_{k=1}^{M} \prod_{j=1, j \neq i}^M \phi_j (\boldsymbol{x})^{\mu_j} },
\end{align}
where $\phi_i$ is the distance function to $\Gamma_i$.
We set $\mu_i = 1$ if $i \in I_D$ and $\mu_i = 2$ if $i \in I_R$, where $I_D \cup I_R = \{1,\dots,M\}$ are index sets indicating on which segments Dirichlet or Robin boundary conditions are prescribed.

\cite{Sukumar21} require an approximate distance function $\phi_i$ in~\eqref{eq:generic_solution_structure_scalar},~\eqref{generic_solution_structure_scalar_weight_functions} that is both a distance function to $\Gamma_i$ in the sense of Definition~\ref{def_distance_function} and normalized with respect to $\Gamma_i$ in the sense of Definition~\ref{def_normalized_function}.
However, it is not always possible to find such a function $\phi_i$ that is $C^1$ everywhere.
Therefore, we propose to use two different functions instead.
The function in~\eqref{eq:generic_solution_structure_scalar} and~\eqref{generic_solution_structure_scalar_weight_functions}, which we still denote as $\phi_i$, is only required to be a distance function to $\Gamma_i$.
The function in the local solution structures $u_i$, which we now denote as $\bar{\phi}_i$, only needs to be a normalized function with respect to $\Gamma_i$.
Below, we show two ways to choose the local solution structures $u_i$.
For comparison,~\eqref{local_solution_structure_Sukumar_ansatz} shows \cite{Sukumar21}'s choice.
\subsubsection{Generalized local solution structure (GLSS) for piecewise $C^1$ boundary}
For pairwise disjoint boundary segments $\Gamma_1,\dots,\Gamma_M$ we use the local solution structure
\begin{align}
  \label{local_solution_structure_Dirichlet_scalar}
  \forall i & \in I_D : & u_i (\boldsymbol{x}) & = \begin{cases}
    g_i (\boldsymbol{x}) + \bar{\phi}_i (\boldsymbol{x}) \tilde{\Psi}_i (\boldsymbol{x}), & \phi_i \text{ has a vanishing gradient}, \\
    g_i (\boldsymbol{x}), & \text{else},
  \end{cases} \\
  \label{local_solution_structure_Robin_scalar}
  \forall i & \in I_R : &
  u_i (\boldsymbol{x}) & = \Psi_{i} (\boldsymbol{x}) - \bar{\phi}_i(\boldsymbol{x}) \nabla \bar{\phi}_i(\boldsymbol{x}) \cdot \nabla \Psi_{i} (\boldsymbol{x}) + \bar{\phi}_i (\boldsymbol{x}) \left( c_i(\boldsymbol{x}) \Psi_i (\boldsymbol{x}) - h_i (\boldsymbol{x}) \right),
\end{align}
where the $\bar{\phi}_i$ are normalized with respect to $\Gamma_i$ and the $\tilde{\Psi}_i$ and $\Psi_i$ are unknown functions to be learned.
The difference between~\eqref{local_solution_structure_Dirichlet_scalar} and \eqref{local_solution_structure_Sukumar_ansatz} is the term $\bar{\phi}_i (\boldsymbol{x}) \tilde{\Psi}_i (\boldsymbol{x})$.
Without it, if $\phi_i$ has a vanishing gradient, we would prescribe both $u_i = g_i$ on $\Gamma_i$ and $\frac{\partial u}{\partial \boldsymbol{n}} = \frac{\partial g_i}{\partial \boldsymbol{n}}$ on $\Gamma_i$, which would overdetermine the problem.
The additional term $\bar{\phi}_i (\boldsymbol{x}) \tilde{\Psi}_i (\boldsymbol{x})$ avoids that by introducing another unknown function to be trained.
In~\cite{Sukumar21}, the $\phi_i$ were required to be normalized with respect to $\Gamma_i$ and thus could not have a vanishing gradient, therefore this problem did not arise.

Intersecting boundary segments as sketched in Figure~\ref{fig:intersecting_boundary} require further modifications to~\eqref{local_solution_structure_Robin_scalar}.
Consider the example of two boundary segments $\Gamma_1$ and $\Gamma_2$ that share one common point $P$, i.e. $\Gamma_1 \cap \Gamma_2 = \{P\}$. 
In that case, the solution structure~\eqref{eq:generic_solution_structure_scalar} would read
\begin{equation}
	u (\boldsymbol{x}) = \frac{\phi_2^{\mu_2} (\boldsymbol{x})}{\phi_1^{\mu_1} (\boldsymbol{x}) + \phi_2^{\mu_2} (\boldsymbol{x})} u_1 (\boldsymbol{x}) + \frac{\phi_1^{\mu_1} (\boldsymbol{x})}{\phi_1^{\mu_1} (\boldsymbol{x}) + \phi_2^{\mu_2} (\boldsymbol{x})} u_2 (\boldsymbol{x}) + \Psi(\boldsymbol{x}) \phi_1 (\boldsymbol{x})^{\mu_1} \phi_2 (\boldsymbol{x})^{\mu_2}.
\end{equation}
Because $\phi_1(P) = \phi_2(P) = 0$, the first two terms divide by zero in $P$.
As pointed out by~\cite{Rvachev01}, we have
\begin{align}
	\lim_{\boldsymbol{x} \in \Gamma_1, \boldsymbol{x} \rightarrow P} u (\boldsymbol{x}) & = \lim_{\boldsymbol{x} \in \Gamma_1, \boldsymbol{x} \rightarrow P} \frac{\phi_2 (\boldsymbol{x})^{\mu_2} }{\phi_2 (\boldsymbol{x})^{\mu_2} } u_1(\boldsymbol{x}) = u_1(P), \\
	\lim_{\boldsymbol{x} \in \Gamma_2, \boldsymbol{x} \rightarrow P} u (\boldsymbol{x}) & = \lim_{\boldsymbol{x} \in \Gamma_2, \boldsymbol{x} \rightarrow P} \frac{\phi_1 (\boldsymbol{x})^{\mu_1} }{\phi_1 (\boldsymbol{x})^{\mu_1} } u_2(\boldsymbol{x}) = u_2(P),
\end{align}
so that the solution structure is only continuous if $u_1(P) = u_2(P)$.
This holds if $u_1$ and $u_2$ both correspond to Dirichlet boundary conditions and the boundary conditions are well-posed, that is, they match continuously and do not prescribe a discontinuous solution.
However, if, say, $u_1$ corresponds to a Robin condition, then $u_1(P) = \Psi_1(P)$ would hold.
Because $\Psi_1$ is learned, $\Psi_1(P)$ will not be equal to $u_2(P)$ and the solution structure will be discontinuous.
To solve this problem, we do not consider $\Psi_1$  as a learned functions directly but instead express it as a function depending on two other learned functions in a way that $\Psi_1(P) = u_2(P)$ is guaranteed to hold.
If $\Gamma_2$ is a Dirichlet segment, then $\Psi_1 (P) = g_2 (P)$ has to be ensured.
However, if $\Gamma_2$ corresponds to a Robin condition, a new unknown function $\Psi_P$ has to be introduced and it must be ensured that $\Psi_1 (P) = \Psi_2 (P) = \Psi_P (P)$ holds.

To make this precise,
let $\Gamma_i$ be a boundary segment with neighbors $\Gamma_a$ and $\Gamma_b$, where $A$ and $B$ are the intersection points.
\begin{figure}[t!]
	\begin{center}
		\begin{tikzpicture}
		\draw[ultra thick, red] (0,0)-- node[above] {$\Gamma_{i}$} (4,0);
		\draw[ultra thick] (0,0)-- node[above left] {$\Gamma_{a}$} (-2,-1);
		\draw[ultra thick] (4,0)-- node[above ] {$\Gamma_{b}$} (6,-1);
		
		\filldraw [red] (0,0) circle (2pt) node[above] {$A$};
		\filldraw [red] (4,0) circle (2pt) node[above] {$B$};
		\end{tikzpicture}
	\end{center}
	\caption{Sketch of intersecting boundary segments.}
	\label{fig:intersecting_boundary}
\end{figure}
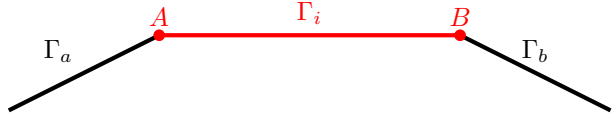
If $\Gamma_i$ is a Robin segment, we define the function $\Psi_i$ as
\begin{equation} \label{eq:Psi_i_GLSS_A_B}
\Psi_i (\boldsymbol{x}) = \frac{\phi_{B} (\boldsymbol{x})}{\phi_{A} (\boldsymbol{x}) + \phi_{B} (\boldsymbol{x})} u_A (\boldsymbol{x}) + \frac{\phi_{A} (\boldsymbol{x})}{\phi_{A} (\boldsymbol{x}) + \phi_{B} (\boldsymbol{x})} u_B (\boldsymbol{x}) + \phi_A (\boldsymbol{x}) \phi_B (\boldsymbol{x}) \bar{\Psi}_i (\boldsymbol{x}),
\end{equation}
where $\phi_A$ and $\phi_B$ are distance functions to $A$ and $B$.
The functions $u_A$ and $u_B$ have to be choosen according to type of boundary conditions prescribed on $\Gamma_a$ and $\Gamma_b$.
If $\Gamma_a$ or $\Gamma_b$ is a Dirichlet segment, we set $u_A = g_a$ or $u_B = g_b$.
In case of a Robin segment, we introduce a new unknown function $\Psi_A$ or $\Psi_B$ and define $u_A=\Psi_A$ or $u_B=\Psi_B$.
The term $\phi_A (\boldsymbol{x}) \phi_B (\boldsymbol{x}) \bar{\Psi}_i (\boldsymbol{x})$ only needs to be included if both $\Gamma_a$ and $\Gamma_b$ are Dirichlet segments.
Function $\bar{\Psi}_i$ is another unknown to be approximated.
We show the complete algorithm for determining local solution structures in section~\ref{sec:algorithms}.

\subsubsection{Orthogonal projections (OP)} \label{sec:orthogonal_projections_scalar}
If all boundary sections with Robin conditions lie in hyperplanes, i.e., for every $\boldsymbol{x} \in \Gamma_i$ the normal vectors $\boldsymbol{n} (\boldsymbol{x})$ are identical, we can use a simpler approach. 
We choose the normalized functions $\bar{\phi}_i$ to be the exact signed distance function to the hyperplane containing $\Gamma_i$.
The local solution structures for Robin conditions are set to
\begin{align}
  \label{local_solution_structure_robin_condition_op}
  \forall i \in I_{R} : \quad &
  u_i (\boldsymbol{x}) = \underbrace{\Psi_i (\mathcal{N} (\boldsymbol{x}; \bar{\phi}_i) )}_{\text{boundary value}} 
  + \underbrace{\bar{\phi}_i (\boldsymbol{x}) f_i(\mathcal{N} (\boldsymbol{x}; \bar{\phi}_i))}_{\text{boundary derivative}}, \\
  \forall i \in I_{R} : \quad &
  f_i (\boldsymbol{x}) = c_i (\boldsymbol{x}) \Psi_i (\boldsymbol{x}) - h_i (\boldsymbol{x}) \quad
  \text{with} \quad \mathcal{N} (\boldsymbol{x}; \bar{\phi}) := \boldsymbol{x} - \bar{\phi} (\boldsymbol{x}) \nabla \bar{\phi} (\boldsymbol{x})
\end{align}
The ansatz uses as a generalized Taylor polynomial expansion by \cite{Rvachev95}.
Here, $\Psi_i$ represents the value and $f_i$ the normal derivative of $u_i$ on $\Gamma_i$.
The concatenations $\Psi_i \circ \mathcal{N} (\,\cdot\, ; \bar{\phi}_i)$ and $f_i \circ \mathcal{N} (\,\cdot\, ; \bar{\phi}_i)$ are the so called normalizer functions of  $\Psi_i$ and $f_i$.

Consider a function $f : \bar{\Omega} \rightarrow \mathbb{R}$ and a boundary segment $\Gamma_i$ with function $\bar{\phi}_i$ that is normalized with respect to $\Gamma_i$.
We call the concatenation $f \circ \mathcal{N} (\,\cdot\, ; \bar{\phi}_i)$ the \emph{normalizer} of $f$, where $\mathcal{N}$ is defined by $\mathcal{N} (\boldsymbol{x}; \bar{\phi}) := \boldsymbol{x} - \bar{\phi} (\boldsymbol{x}) \nabla \bar{\phi} (\boldsymbol{x})$.
Near the boundary, the normalizer behaves like a constant in the normal direction of $\Gamma_i$ so that
\begin{equation}
	\label{eq:property_normalizer_constant_in_normal_direction}
	\forall \boldsymbol{x} \in \Gamma_i : \quad f \left(\mathcal{N} \left(\boldsymbol{x}; \bar{\phi}_i\right)\right) = f(\boldsymbol{x}), \quad \frac{\partial f \left(\mathcal{N} \left(\boldsymbol{x}; \bar{\phi}_i\right)\right)}{\partial \boldsymbol{\nu}} = 0
\end{equation}
holds if $\bar{\phi}_i$ is normalized with respect to $\Gamma_i$.
This property assures that the generalized Taylor polynomial in~\eqref{local_solution_structure_robin_condition_op} satisfies the prescribed boundary conditions.

The first term in~\eqref{local_solution_structure_robin_condition_op} determines the value on $\Gamma_i$ but has zero derivative whilst the second term is zero on $\Gamma_i$ but has non-zero derivative equal to the Robin condition.
Since the exact signed distance function to $\Gamma_i$ is normalized with respect to $\Gamma_i$, property~\eqref{eq:property_normalizer_constant_in_normal_direction} holds.
Thus, $\mathcal{N} (\,\cdot\, ; \bar{\phi}_i)$ is simply the orthogonal projection, mapping its argument onto the corresponding hyperplane, as visualized in Figure~\ref{fig:n}.

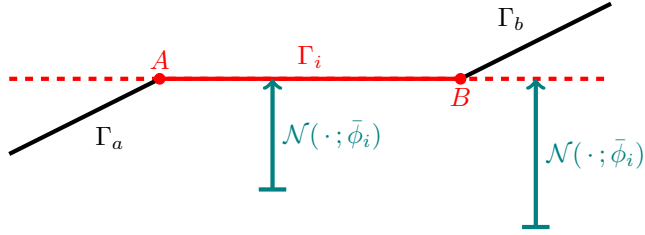
\begin{figure}[th]
	\begin{center}
		\begin{tikzpicture}
			\draw[ultra thick, dashed, red] (-2,0) -- (6,0);
			\draw[ultra thick, red] (0,0)-- node[above] {$\Gamma_{i}$} (4,0);
			\draw[ultra thick] (0,0)-- node[below right] {$\Gamma_{a}$} (-2,-1);
			\draw[ultra thick] (4,0)-- node[above left] {$\Gamma_{b}$} (6,1);
			
			\filldraw [red] (0,0) circle (2pt) node[above] {$A$};
			\filldraw [red] (4,0) circle (2pt) node[below] {$B$};
			
			\draw[ultra thick, |->, teal] (5,-2)-- node[right] {$\mathcal{N} (\, \cdot \, ; \bar{\phi}_i)$}  (5,0);

			\draw[ultra thick, |->, teal] (1.5,-1.5)-- node[right] {$\mathcal{N} (\, \cdot \, ; \bar{\phi}_i)$}  (1.5,0);
			
		\end{tikzpicture}
	\end{center}
	\caption{If $\bar{\phi}_i$ is the exact signed distance function to the hyperplane that $\Gamma_i$ lies on (indicated by the dotted line), $\mathcal{N} (\, \cdot \, ; \bar{\phi}_i)$ maps its argument onto the hyperplane, which is illustrated by the arrows.}
	\label{fig:n}
\end{figure}

Therefore, each $\Psi_i$ is only evaluated on its corresponding hyperplane and we can set
\begin{equation} \label{eq:scalar_op_psi_i_dirichlet_fitting}
  \forall i \in I_{R} : \quad \Psi_i (\boldsymbol{x}) = g(\boldsymbol{x}) + \tilde{\Psi} (\boldsymbol{x}) \prod_{k \in I_D} \phi_k (\boldsymbol{x}),
\end{equation}
where $g$ is a function satisfying all Dirichlet conditions.
This avoids discontinuities of the transfinite interpolant at intersection points, and reduced the computational effort, since only the single function $\tilde{\Psi}$ needs to be trained.

\subsection{Exact boundary conditions for systems of partial differential equations}\label{subsec:pde_systems}
Consider a system of differential equations with solution $\boldsymbol{u} : \mathbb{R}^2 \supset \Omega \rightarrow \mathbb{R}^n$ and boundary conditions prescribed on segments $\Gamma_1, \dots, \Gamma_M \subset \partial \Omega$ with each $\Gamma_i$ being $C^1$.
For $i \in \{1,\dots,M\}$ and $\boldsymbol{x} \in \Gamma_i$ let $\boldsymbol{b}_i^{(1)} (\boldsymbol{x}), \dots, \boldsymbol{b}_i^{(n)} (\boldsymbol{x})$ be a basis of $\mathbb{R}^n$.
Let $\mathit{IJ}_D, \mathit{IJ}_R \subset I \times J := \{1,\dots,M\} \times \{1,\dots,n\}$ be index sets such that
\begin{align}
  \label{eq:Dirichlet_boundary_condition_system_case}
  \forall (i,j) & \in \mathit{IJ}_D : \forall \boldsymbol{x} \in \Gamma_i: &
  \boldsymbol{b}_i^{(j)} (\boldsymbol{x}) \cdot \boldsymbol{u} (\boldsymbol{x}) & = g_i^{(j)} (\boldsymbol{x}), \\
  \label{eq:Robin_boundary_condition_system_case}
  \forall (i,j) & \in \mathit{IJ}_R : \forall \boldsymbol{x} \in \Gamma_i: &
  \frac{\partial \left(\boldsymbol{b}_{i}^{(j)} (\boldsymbol{x}) \cdot \boldsymbol{u} (\boldsymbol{x}) \right)}{\partial \boldsymbol{n}} + \boldsymbol{c}_i^{(j)} (\boldsymbol{x}) \cdot \boldsymbol{u} (\boldsymbol{x}) & = h_i^{(j)} (\boldsymbol{x}),
\end{align}
Without loss of generality, we assume that for every $(i,j) \in \mathit{IJ}_R$ and $\boldsymbol{x} \in \Gamma_i$ the $c_i^{(j)} (\boldsymbol{x})$ lies in $\Span(\{ \boldsymbol{b}_i^{(k)} (\boldsymbol{x}) \, | \, (i,k) \not \in \mathit{IJ}_D \})$.
The generic solution structure now becomes
\begin{align}
  \label{generic_solution_structure_system}
  & \boldsymbol{u} (\boldsymbol{x}) = \sum_{i=1}^{M} w_i (\boldsymbol{x}) \boldsymbol{u}_i (\boldsymbol{x})
  \enspace + \enspace [\Psi^{(1)}(\boldsymbol{x}), \dots, \Psi^{(n)}(\boldsymbol{x})]^T \prod_{i=1}^{M} \phi_i (\boldsymbol{x})^{\mu_i}, \\
  & \mu_i = \begin{cases}
    2, & \exists j : (i,j) \in \mathit{IJ}_R, \\
    1, & \text{else,}
  \end{cases}
\end{align}
with weights given in \eqref{generic_solution_structure_scalar_weight_functions}. 
Functions $\boldsymbol{u}_i$ are expressed as a linear combination of the basis functions
\begin{equation} \label{local_solution_structure_system_ansatz_combination_of_basis_functions}
  \forall i \in I : \quad \boldsymbol{u}_i (\boldsymbol{x}) = \sum_{j=1}^{n} \boldsymbol{b}_i^{(j)} (\boldsymbol{x}) u_i^{(j)} (\boldsymbol{x}).
\end{equation}
\begin{proposition}\label{prop:bc-system}
If the vectors $\boldsymbol{b}_i^{(1)} (\boldsymbol{x}), \dots, \boldsymbol{b}_i^{(n)} (\boldsymbol{x})$ form a basis of $\mathbb{R}^n$ for every $\boldsymbol{x} \in N(\Gamma_i)$ where $N(\Gamma_i)$ is an open neighborhood of $\Gamma_i$, we have
  \begin{equation}
    \label{eq:prop_a}
    \forall (i,j)  \in I \times J : \  \forall \boldsymbol{x} \in \Gamma_i :  \
    \boldsymbol{b}_i^{(j)} (\boldsymbol{x}) \cdot \boldsymbol{u} (\boldsymbol{x})  = u_i^{(j)} (\boldsymbol{x}).
    \end{equation}
    Further,
    \begin{equation}
    \label{eq:prop_b}
    \forall (i,j)  \in I \times J  : \  \forall \boldsymbol{x}  \in \Gamma_i : \
    \frac{\partial \left( \boldsymbol{b}_i^{(j)} (\boldsymbol{x}) \cdot \boldsymbol{u} (\boldsymbol{x}) \right)}{\partial \boldsymbol{n}}  = \frac{\partial u_i^{(j)} (\boldsymbol{x})}{\partial \boldsymbol{n}},
  \end{equation} 
  if $\phi_i^{\mu_i}$ has  a vanishing gradient.
\end{proposition} 

\begin{proof}
The proof is shown in section~\ref{sec:prop-proof}.
\end{proof}

Therefore, every  $u_i^{(j)}$ has to satisfy the corresponding boundary condition.
For Dirichlet conditions this is achieved by setting
\begin{equation} \label{eq:ansatz_local_solution_structure_dirichlet_system_glss_op}
  \forall (i,j) \in \mathit{IJ}_D : \quad u_i^{(j)} (\boldsymbol{x}) = \begin{cases}
    g_i^{(j)} (\boldsymbol{x}) + \bar{\phi}_i (\boldsymbol{x}) \tilde{\Psi}_i^{(j)} (\boldsymbol{x}), & \phi_{i}^{\mu_i} \text{ has a vanishing gradient}, \\
    g_i^{(j)} (\boldsymbol{x}), & \text{else,}
  \end{cases}
\end{equation}
for both GLSS and OP.
However, the two approaches differ in their treatment of Robin conditions.

\subsubsection{GLSS}
If all boundary segments $\Gamma_1, \dots, \Gamma_M$ are pairwise disjoint, we set
\begin{align}
  \label{eq:ansatz_local_solution_structure_system_glss}
  u_i^{(j)} (\boldsymbol{x}) & = \Psi_i^{(j)} (\boldsymbol{x}) - \bar{\phi}_i (\boldsymbol{x}) \nabla \bar{\phi}_i (\boldsymbol{x}) \cdot \nabla \Psi_i^{(j)} (\boldsymbol{x}) + \bar{\phi}_i (\boldsymbol{x}) f_i^{(j)}(\boldsymbol{x}), \\
  \label{eq:ansatz_f_i_j}
  f_i^{(j)} (\boldsymbol{x}) & = \boldsymbol{c}_i^{(j)} (\boldsymbol{x}) \cdot \sum_{k=1, (i,k) \not\in \mathit{IJ}_D}^{n} \boldsymbol{b}_i^{(k)} (\boldsymbol{x}) \Psi_i^{(k)} (\boldsymbol{x}) \enspace - \enspace h_i^{(j)} (\boldsymbol{x}).
\end{align}
for $(i,j) \in \mathit{IJ}_R$.
For $(i,j) \in I \times J \setminus (\mathit{IJ}_D \cup \mathit{IJ}_R)$, we define
$u_i^{(j)} (\boldsymbol{x}) = \Psi_i^{(j)} (\boldsymbol{x})$.
As above, the functions $\Psi_i^{(j)}$ have to be modified, if $\Gamma_i$ has intersection points with other segments.

\paragraph{Intersecting boundary segments.}
We generalize our approach to the system case and let
\begin{equation}
	\begin{split}
  \Psi_i^{(j)} (\boldsymbol{x}) = \boldsymbol{b}_i^{(j)} (\boldsymbol{x}) \cdot \left( \frac{\phi_{B} (\boldsymbol{x})}{\phi_{A} (\boldsymbol{x}) + \phi_{B} (\boldsymbol{x})} \boldsymbol{u}_A (\boldsymbol{x}) + \frac{\phi_{A} (\boldsymbol{x})}{\phi_{A} (\boldsymbol{x}) + \phi_{B} (\boldsymbol{x})} \boldsymbol{u}_B (\boldsymbol{x}) \right)\\ + \phi_A (\boldsymbol{x}) \phi_B (\boldsymbol{x}) \bar{\Psi}_i^{(j)} (\boldsymbol{x}).
\end{split}
\end{equation}
The construction of the functions $\boldsymbol{u}_A$ and $\boldsymbol{u}_B$ is more difficult than in the scalar case.
We demonstrate how to do this with an example.
Consider two segments $\Gamma_1$ and $\Gamma_2$ with intersection point $P$ and assume for simplicity that $u \in \mathbb{R}^3$.
Let Dirichlet conditions be prescribed on $\Gamma_1$ with respect to the basis vectors $\boldsymbol{b}_1^{(1)}(P)=(1,0,0)^T$ and $\boldsymbol{b}_1^{(2)}(P)=(0,1,0)^T$ and on $\Gamma_2$ with respect to the vector $\boldsymbol{b}_2^{(1)}(P)=(1,1,0)^T$.
These three basis vectors span a two-dimensional subspace with basis $(1,0,0)^T, (0,1,0)^T$.
We define $\boldsymbol{u}_P$ as a linear combination of this basis and an unknown component acting on the orthogonal complement, i.e.
\begin{equation}
  \boldsymbol{u}_P (\boldsymbol{x}) = g_P^{(1)} \begin{pmatrix}
    1\\0\\0
  \end{pmatrix}
  + g_P^{(2)} \begin{pmatrix}
    0\\1\\0
  \end{pmatrix}
  + \Psi_P^{(3)} (\boldsymbol{x}) \begin{pmatrix}
    0 \\ 0 \\ 1
  \end{pmatrix}.
\end{equation}
Note that the constants $g_P^{(1)}$ and $g_P^{(2)}$ have to be chosen such that $\boldsymbol{u}_P$ satisfies all Dirichlet conditions prescribed in $P$.
A complete algorithm can be found in section~\ref{sec:algorithms}.

\subsubsection{OP}
If all boundary segments $\Gamma_i$ lie in hyperplanes and $\boldsymbol{b}^{(j)} := \boldsymbol{b}_1^{(j)} = \dots = \boldsymbol{b}_M^{(j)}$ holds for every $j=1,\dots,n$, the global solution structure, given by \eqref{generic_solution_structure_system} and \eqref{local_solution_structure_system_ansatz_combination_of_basis_functions}, simplifies to
\begin{equation} \label{eq:global_solution_structure_op_system_simplified}
  \boldsymbol{u} (\boldsymbol{x}) = \sum_{j=1}^{n} \boldsymbol{b}^{(j)} (\boldsymbol{x}) \sum_{i=1}^{M} w_i (\boldsymbol{x}) u_i^{(j)} (\boldsymbol{x}) \enspace + \enspace [\Psi^{(1)} (\boldsymbol{x}), \dots, \Psi^{(n)} (\boldsymbol{x})]^T \prod_{i=1}^{M} \phi_i (\boldsymbol{x})^{\mu_i}.
\end{equation}
We choose the local solution structures as
\begin{equation} \label{eq:ansatz_local_solution_structure_op_system}
  u_i^{(j)} (\boldsymbol{x}) = \begin{cases}
    \begin{cases}
      g_i^{(j)} (\boldsymbol{x}) + \bar{\phi}_i (\boldsymbol{x}) \tilde{\Psi}_i (\boldsymbol{x}), & \phi_i \text{ has a vanishing gradient} \\
      g_i^{(j)} (\boldsymbol{x}), & \text{else}
    \end{cases}, & (i,j) \in \mathit{IJ}_D, \\
    \bar{\Psi}^{(j)} (\mathcal{N}(\boldsymbol{x};\bar{\phi}_i)) + \bar{\phi}_i (\boldsymbol{x}) f_i^{(j)} (\mathcal{N} (\boldsymbol{x}; \bar{\phi}_i)), & (i,j) \in \mathit{IJ}_R, \\
    \bar{\Psi}^{(j)} (\boldsymbol{x}), & \text{else}
  \end{cases} 
\end{equation}
The functions $f_i^{(j)}$ are defined as
\begin{equation}
  f_i^{(j)} (\boldsymbol{x}) = \boldsymbol{c}_i^{(j)} (\boldsymbol{x}) \cdot \sum_{k=1, (i,k) \not \in \mathit{IJ}_D}^{n} \bar{\Psi}^{(k)} (\boldsymbol{x}) \boldsymbol{b}^{(k)} (\boldsymbol{x}) \enspace - \enspace h_i^{(j)} (\boldsymbol{x}),
\end{equation}
and the $\bar{\Psi}^{(j)}$ are defined as
\begin{equation} \label{eq:system_op_bar_psi_j_tilde_psi_j}
  \bar{\Psi}^{(j)} (\boldsymbol{x}) = g^{(j)} (\boldsymbol{x}) + \tilde{\Psi}^{(j)} (\boldsymbol{x}) \prod_{i=1, (i,j) \in \mathit{IJ}_D}^{M} \phi_i (\boldsymbol{x}).
\end{equation}
Each function $g^{(j)}$ is chosen in a way that it satisfies all Dirichlet conditions prescribed with respect to the basis vector $\boldsymbol{b}^{(j)}$ and $\tilde{\Psi}^{(j)}$ is an unknown function to be approximated.

\begin{theorem}\label{thm:bcsatisfied}
  The derived solution structure satisfies the boundary conditions \eqref{eq:Dirichlet_boundary_condition_system_case} and \eqref{eq:Robin_boundary_condition_system_case} for both the GLSS and OP approach.
\end{theorem}
\begin{proof}
The proof can be found in section~\ref{sec:appendix_proof}.
\end{proof}

\subsection{Algorithms}\label{sec:algorithms}
For scalar problems, the full algorithm to create the solution structure is shown in Algorithm~\ref{alg:intersecting_boundary_segments_scalar}.
The algorithm to construct the solution structures for the system case is shown in Algorithm~\ref{alg:intersecting_boundary_segments_system_part_1} and Algorithm~\ref{alg:intersecting_boundary_segments_system_part_2}.
\begin{algorithm}[th]
	\caption{Algorithm to determine local solution structure $u_i$ for for GLSS}\label{alg:intersecting_boundary_segments_scalar}
	\begin{algorithmic}
		\For{$i=1,\dots,M$}
		\If{$\Gamma_i$ is a Dirichlet segment}
		\If{$\phi_i$ has a vanishing gradient} 
		\State $\tilde{\Psi}_i$ is considered as an unknown function
		\State $u_i \gets \left( \boldsymbol{x} \mapsto g_i (\boldsymbol{x}) + \bar{\phi}_i (\boldsymbol{x}) \tilde{\Psi}_i (\boldsymbol{x}) \right)$
		\Else     
		\State $u_i \gets \left( \boldsymbol{x} \mapsto g_i (\boldsymbol{x}) \right)$
		\EndIf
		\Else
		\If{$\Gamma_i$ has no intersection with any other segment} 
		\State $\Psi_i$ is considered as an unknown function
		\ElsIf{$\Gamma_i$ has neighboring boundary segments $\Gamma_a$ and $\Gamma_b$ with corresponding intersecting points $A$ and $B$ as in Fig. \ref{fig:intersecting_boundary}} 
		\For{$(p,P) \in \{(a,A),(b,B)\}$}
		\If{$\Gamma_p$ is Dirichlet segment}
		\State $u_P \gets (\boldsymbol{x} \mapsto g_p(\boldsymbol{x}))$ 
		\ElsIf{$\Gamma_p$ is Robin segment}
		\State $\Psi_P$ is considered as an unknown function
		\State $u_P \gets (\boldsymbol{x} \mapsto \Psi_P(\boldsymbol{x}))$ 
		\EndIf
		\EndFor
		\State $\bar{\Psi}_i$ is considered as an unknown function
		\State $\Psi_i \gets \left( \boldsymbol{x} \mapsto \frac{\phi_{B} (\boldsymbol{x})}{\phi_{A} (\boldsymbol{x}) + \phi_{B} (\boldsymbol{x})} u_A (\boldsymbol{x}) + \frac{\phi_{A} (\boldsymbol{x})}{\phi_{A} (\boldsymbol{x}) + \phi_{B} (\boldsymbol{x})} u_B (\boldsymbol{x}) + \phi_A (\boldsymbol{x}) \phi_B (\boldsymbol{x}) \bar{\Psi}_i (\boldsymbol{x}) \right)$
		\EndIf
		\State $u_i \gets \left( \boldsymbol{x} \mapsto \Psi_{i} (\boldsymbol{x}) - \bar{\phi}_i(\boldsymbol{x}) \nabla \bar{\phi}_i(\boldsymbol{x}) \cdot \nabla \Psi_{i} (\boldsymbol{x}) + \bar{\phi}_i (\boldsymbol{x}) \left( c_i(\boldsymbol{x}) \Psi_i (\boldsymbol{x}) - h_i (\boldsymbol{x}) \right) \right)$
		\EndIf
		\EndFor    
	\end{algorithmic}
\end{algorithm}
\begin{algorithm}[th]
	\caption{Algorithm to determine the local solution structures $u_i^{(j)}$ for GLSS (Part 1)}\label{alg:intersecting_boundary_segments_system_part_1}
	\begin{algorithmic}[1]
		\State $\boldsymbol{\mathit{IP}} \gets \{\}$
		\For{$i=1,\dots,M$}
		\If{$\Gamma_i$ has common points $A$ and $B$ with other segments \textbf{and} $\exists j \in \{1,\dots,n\} : \boldsymbol{b}_i^{(j)}$ does not correspond to a Dirichlet condition}
		\State $\boldsymbol{\mathit{IP}} \gets \boldsymbol{\mathit{IP}} \cup \{A,B\}$
		\EndIf
		\EndFor \Comment{Now, $\boldsymbol{\mathit{IP}}$ contains all intersection-points that need to be taken extra care of}
		\For{every $P \in \boldsymbol{\mathit{IP}}$}
		\State $\boldsymbol{B} \times \boldsymbol{G} \gets \{\} \times \{\}$
		\For{every $\Gamma_i$ with $P \in \Gamma_i$}
		\For{every $j \in \{1,\dots,n\}$ with $\boldsymbol{b}_i^{(j)}$ belonging to a Dirichlet condition}
		\If{$\dim(\Span(\boldsymbol{B}\cup\{\boldsymbol{b}_i^{(j)}(P)\})) > \dim(\Span(\boldsymbol{B}))$}
		\State $\boldsymbol{B} \times \boldsymbol{G} \gets \boldsymbol{B} \times \boldsymbol{G} \cup \{ (\boldsymbol{b}_i^{(j)} (P), g_i^{(j)} (P)) \}$
		\EndIf
		\EndFor
		\EndFor
		\State $\boldsymbol{b}_P^{(1)}, \dots, \boldsymbol{b}_P^{(\dim(\Span(\boldsymbol{B})))} \gets \text{basis of } \Span(\boldsymbol{B})$
		\State $A \gets \text{zeros} (\dim(\Span(\boldsymbol{B})),\dim(\Span(\boldsymbol{B})))$
		\State $b \gets \text{zeros} (\dim(\Span(\boldsymbol{B})),1)$
		\For{$k,(\boldsymbol{b},g) \in \text{enumerate}(\boldsymbol{B} \times \boldsymbol{G})$}
		\State $A[k,:] \gets \boldsymbol{b}^T (\boldsymbol{b}_P^{(1)}, \dots, \boldsymbol{b}_P^{(\dim(\Span(\boldsymbol{B})))})$
		\State $b[k] \gets g$
		\EndFor
		\State $(g_P^{(1)}, \dots, g_P^{(\dim(\Span(\boldsymbol{B})))})^T \gets \text{solve}(A,b)$
		\State $\boldsymbol{b}_P^{(\dim(\Span(\boldsymbol{B}))+1)}, \dots, \boldsymbol{b}_P^{(n)} \gets \text{basis of orthogonal complement of } \Span(\boldsymbol{B})$
		\State $\Psi_P^{(\dim(\Span(\boldsymbol{B}))+1)}, \dots, \Psi_P^{(n)}$ are considered as unknown functions
		\State $\boldsymbol{u}_P \gets \left( \boldsymbol{x} \mapsto \sum_{j=1}^{\dim(\Span(\boldsymbol{B}))} \boldsymbol{b}_P^{(j)} g_P^{(j)} + \sum_{j=\dim(\Span(\boldsymbol{B}))+1}^{n} \boldsymbol{b}_P^{(j)} \Psi_P^{(j)} (\boldsymbol{x}) \right)$
		\EndFor
		\algstore{algorithm_system_of_equations}
	\end{algorithmic}
\end{algorithm}

\begin{algorithm}
	\caption{Algorithm to determine the local solution structures $u_i^{(j)}$ for GLSS (part 2)}\label{alg:intersecting_boundary_segments_system_part_2}
	\begin{algorithmic}[1]
		\algrestore{algorithm_system_of_equations}
		\For{$i=1,\dots,M$}
		\For{$j=1,\dots,n$}
		\If{$\boldsymbol{b}_i^{(j)}$ corresponds to a Dirichlet condition}
		\If{$\phi_i^{\mu_i}$ has a non-zero gradient on $\Gamma_i$}
		\State $u_i^{(j)} \gets \left( \boldsymbol{x} \mapsto g_i^{(j)}(\boldsymbol{x}) \right)$
		\Else
		\State $\tilde{\Psi}_i^{(j)}$ is considered as unknown function
		\State $u_i^{(j)} \gets \left( \boldsymbol{x} \mapsto g_i^{(j)}(\boldsymbol{x}) + \bar{\phi}_i (\boldsymbol{x}) \tilde{\Psi}_i^{(j)} (\boldsymbol{x}) \right)$
		\EndIf
		\Else
		\If{$\Gamma_i$ has common points $A$ and $B$ with other segments}
		\State $\Psi_i^{(j)} \gets \quad \boldsymbol{x} \mapsto \boldsymbol{b}_i^{(j)} (\boldsymbol{x}) \cdot \left( \frac{\phi_B(\boldsymbol{x})}{\phi_A(\boldsymbol{x}) + \phi_B(\boldsymbol{x})} \boldsymbol{u}_A (\boldsymbol{x}) +  \frac{\phi_A(\boldsymbol{x})}{\phi_A(\boldsymbol{x}) + \phi_B(\boldsymbol{x})} \boldsymbol{u}_B (\boldsymbol{x}) \right) $
		\State \hspace{1.5cm}$+ \phi_A (\boldsymbol{x}) \phi_B (\boldsymbol{x}) \bar{\Psi}_i^{(j)} (\boldsymbol{x})$
		\Else
		\State $\Psi_i^{(j)}$ is considered as unknown function 
		\EndIf
		\If{$\boldsymbol{b}_i^{(j)}$ belongs to a Robin condition}
		\State $u_i^{(j)} \gets \quad \boldsymbol{x} \mapsto \Psi_i^{(j)} (\boldsymbol{x}) - \bar{\phi}_i (\boldsymbol{x}) \nabla \bar{\phi}_i (\boldsymbol{x}) \cdot \nabla \Psi_i^{(j)} (\boldsymbol{x})$
		\State \hspace{1.5cm}$+ \bar{\phi}_i (\boldsymbol{x}) \left( \boldsymbol{c}_i^{(j)} (\boldsymbol{x}) \cdot \sum\limits_{k=1, (i,k) \not \in \mathit{IJ}_D}^{n} \boldsymbol{b}_i^{(k)} (\boldsymbol{x}) \Psi_i^{(k)} (\boldsymbol{x})  -  h_i^{(j)} (\boldsymbol{x}) \right)$
		\Else
		\State $u_i^{(j)} \gets \left( \boldsymbol{x} \mapsto \Psi_i^{(j)} (\boldsymbol{x}) \right)$
		\EndIf
		\EndIf
		\EndFor
		\EndFor
	\end{algorithmic}
\end{algorithm}

\subsection{Proofs}
Here we provide the proofs of  proposition~\ref{prop:bc-system} and theorem~\ref{thm:bcsatisfied}.

\subsubsection{Proof of proposition~\ref{prop:bc-system}}\label{sec:prop-proof}
Let $(i,j) \in I \times J$ and $\boldsymbol{x} \in \Gamma_i$
Since $\phi_i$ is a distance function with respect to $\Gamma_i$, we have $\phi_i (\boldsymbol{x}) = 0$.
Therefore, the generic solution structure $\boldsymbol{u}$ given by~\eqref{generic_solution_structure_system} is equal to $\boldsymbol{u}_i$.
Using this and substituting $\boldsymbol{u}_i$ by~\eqref{local_solution_structure_system_ansatz_combination_of_basis_functions} gets us
\begin{equation}
	\boldsymbol{b}_i^{(j)} (\boldsymbol{x}) \cdot \boldsymbol{u} (\boldsymbol{x}) = \boldsymbol{b}_i^{(j)} (\boldsymbol{x}) \cdot \boldsymbol{u}_i (\boldsymbol{x}) =  \boldsymbol{b}_i^{(j)} (\boldsymbol{x}) \cdot \sum_{k=1}^{n} \boldsymbol{b}_i^{(k)} (\boldsymbol{x}) u_i^{(k)} (\boldsymbol{x}) = u_i^{(j)} (\boldsymbol{x}),
\end{equation}
where we abused the fact that $\boldsymbol{b}_i^{(1)} (\boldsymbol{x}), \dots, \boldsymbol{b}_i^{(n)} (\boldsymbol{x})$ is an orthonormal basis.
This proves~\eqref{eq:prop_a}.
To prove~\eqref{eq:prop_b}, let $(\tilde{i},\tilde{j}) \in I \times J$ and let $\phi_{\tilde{i}}^{\mu_{\tilde{i}}}$ have a vanishing gradient and $\boldsymbol{x} \in \Gamma_{\tilde{i}}$.
We substitute the ansatz for $\boldsymbol{u}$ given by~\eqref{generic_solution_structure_system} into $\frac{\partial \left( \boldsymbol{b}_{\tilde{i}}^{(\tilde{j})} (\boldsymbol{x}) \cdot \boldsymbol{u} (\boldsymbol{x}) \right)}{\partial \boldsymbol{n}}$ and calculate
\begin{multline}
	\frac{\partial \left( \boldsymbol{b}_{\tilde{i}}^{(\tilde{j})} (\boldsymbol{x}) \cdot \boldsymbol{u} (\boldsymbol{x}) \right)}{\partial \boldsymbol{n}} \\
	= \frac{\partial \left( \boldsymbol{b}_{\tilde{i}}^{(\tilde{j})} (\boldsymbol{x}) \cdot \left(\sum_{i=1}^{M} w_i (\boldsymbol{x}) \boldsymbol{u}_i (\boldsymbol{x}) + [\Psi^{(1)}(\boldsymbol{x}), \dots, \Psi^{(n)}(\boldsymbol{x})]^T \prod_{i=1}^{M} \phi_i (\boldsymbol{x})^{\mu_i}\right) \right)}{\partial \boldsymbol{n}} \\
	= \frac{\partial \left( \sum_{i=1}^{M} w_i (\boldsymbol{x}) \boldsymbol{b}_{\tilde{i}}^{(\tilde{j})} (\boldsymbol{x}) \cdot \boldsymbol{u}_i (\boldsymbol{x}) \right)}{\partial \boldsymbol{n}} \\
	+ \frac{\partial \left( \boldsymbol{b}_{\tilde{i}}^{(\tilde{j})} (\boldsymbol{x}) \cdot [\Psi^{(1)}(\boldsymbol{x}), \dots, \Psi^{(n)}(\boldsymbol{x})]^T \prod_{i=1}^{M} \phi_i (\boldsymbol{x})^{\mu_i}\right)}{\partial \boldsymbol{n}}.
\end{multline}
Using product rule, we find that the second term is zero because $\phi_{\tilde{i}}^{\mu_{\tilde{i}}} (\boldsymbol{x}) = 0$ and $\frac{\partial \phi_{\tilde{i}}^{\mu_{\tilde{i}}} (\boldsymbol{x})}{\partial \boldsymbol{n}} = 0$.
The remaining first term can be simplified as follows
\begin{multline}
	\label{eq:proof_prop_robin_condition_u_tilde_i}
	\frac{\partial \left( \boldsymbol{b}_{\tilde{i}}^{(\tilde{j})} (\boldsymbol{x}) \cdot \boldsymbol{u} (\boldsymbol{x}) \right)}{\partial \boldsymbol{n}}  
	= \frac{\partial \left( \sum_{i=1}^{M} w_i (\boldsymbol{x}) \boldsymbol{b}_{\tilde{i}}^{(\tilde{j})} (\boldsymbol{x}) \cdot \boldsymbol{u}_i (\boldsymbol{x}) \right)}{\partial \boldsymbol{n}}\\
	= \sum_{i=1}^{M} \biggl( 
	\underbrace{\frac{\partial w_i(\boldsymbol{x})}{\partial \boldsymbol{n}}}_{=0} \boldsymbol{b}_{\tilde{i}}^{(\tilde{j})} (\boldsymbol{x}) \cdot \boldsymbol{u}_i (\boldsymbol{x}) 
	+ \underbrace{w_i(\boldsymbol{x})}_{=\begin{cases} 1, & \tilde{i} = i \\ 0, & \text{else} \end{cases}} \frac{\partial \boldsymbol{b}_{\tilde{i}}^{(\tilde{j})} (\boldsymbol{x}) \cdot \boldsymbol{u}_i (\boldsymbol{x})}{\partial \boldsymbol{n}} 
	\biggr)\\
	= \frac{\partial \left( \boldsymbol{b}_{\tilde{i}}^{(\tilde{j})} (\boldsymbol{x}) \cdot \boldsymbol{u}_{\tilde{i}} (\boldsymbol{x}) \right)}{\partial \boldsymbol{n}}.
\end{multline}
To show that $\frac{\partial w_i(\boldsymbol{x})}{\partial \boldsymbol{n}} = 0$, we first consider the case that $i \neq \tilde{i}$ and calculate
\begin{equation}
	\frac{\partial w_i(\boldsymbol{x})}{\partial \boldsymbol{n}}
	= \frac{\partial \frac{ \prod_{j=1, j \neq i}^M \phi_j (\boldsymbol{x})^{\mu_j} }{ \sum_{k=1}^{M} \prod_{j=1, j \neq k}^M \phi_j (\boldsymbol{x})^{\mu_j} }}{\partial \boldsymbol{n}}
	= \frac{\partial \phi_{\tilde{i}}^{\mu_{\tilde{i}}} (\boldsymbol{x}) \frac{ \prod_{j=1, j \neq i, j \neq \tilde{i}}^M \phi_j (\boldsymbol{x})^{\mu_j} }{ \sum_{k=1}^{M} \prod_{j=1, j \neq k}^M \phi_j (\boldsymbol{x})^{\mu_j} }}{\partial \boldsymbol{n}} 
	\eqqcolon \frac{\partial \phi_{\tilde{i}}^{\mu_{\tilde{i}}} (\boldsymbol{x}) f(\boldsymbol{x}) }{\partial \boldsymbol{n}},
\end{equation}
where we introduce the function $f$ for simplification.
We use product rule and get
\begin{equation}
	\frac{\partial \phi_{\tilde{i}}^{\mu_{\tilde{i}}} (\boldsymbol{x}) f(\boldsymbol{x}) }{\partial \boldsymbol{n}}
	= \underbrace{\frac{\partial \phi_{\tilde{i}}^{\mu_{\tilde{i}}} (\boldsymbol{x}) }{\partial \boldsymbol{n}}}_{=0}
	f (\boldsymbol{x})
	+ \underbrace{\phi_{\tilde{i}}^{\mu_{\tilde{i}}} (\boldsymbol{x})}_{=0} 
	\frac{\partial f(\boldsymbol{x})}{\partial \boldsymbol{n}}
	= 0.
\end{equation}
We now consider the case that $i = \tilde{i}$ and calculate
\begin{multline}
	\frac{\partial w_i(\boldsymbol{x})}{\partial \boldsymbol{n}}
	= \frac{\partial \frac{ \prod_{j=1, j \neq i}^M \phi_j (\boldsymbol{x})^{\mu_j} }{ \sum_{k=1}^{M} \prod_{j=1, j \neq k}^M \phi_j (\boldsymbol{x})^{\mu_j} }}{\partial \boldsymbol{n}}
	= \frac{\partial \frac{ \prod_{j=1, j \neq i}^M \phi_j (\boldsymbol{x})^{\mu_j} }{ \prod_{j=1, j \neq i}^M \phi_j (\boldsymbol{x})^{\mu_j} + \sum_{k=1, k\neq i}^{M} \prod_{j=1, j \neq k}^M \phi_j (\boldsymbol{x})^{\mu_j} }}{\partial \boldsymbol{n}} \\
	\eqqcolon \frac{\partial \frac{u(\boldsymbol{x})}{u(\boldsymbol{x}) + v (\boldsymbol{x})} }{\partial \boldsymbol{n}}
\end{multline}
Next we use the quotient rule to find
\begin{equation}
	\frac{\partial \frac{u(\boldsymbol{x})}{u(\boldsymbol{x}) + v (\boldsymbol{x})} }{\partial \boldsymbol{n}}
	= \frac{ 
		\frac{\partial u(\boldsymbol{x})}{\partial \boldsymbol{n}} (u(\boldsymbol{x}) + v (\boldsymbol{x}))
		- u(\boldsymbol{x}) \frac{\partial ( u(\boldsymbol{x}) + v (\boldsymbol{x})) }{\partial \boldsymbol{n}}
	}{ (u(\boldsymbol{x}) + v (\boldsymbol{x}))^2 }
	= \frac{ 
		\frac{\partial u(\boldsymbol{x})}{\partial \boldsymbol{n}} v (\boldsymbol{x})
		- u(\boldsymbol{x}) \frac{\partial v (\boldsymbol{x})}{\partial \boldsymbol{n}}
	}{ (u(\boldsymbol{x}) + v (\boldsymbol{x}))^2 } = 0,
\end{equation}
using $v(\boldsymbol{x}) = \frac{\partial v(\boldsymbol{x})}{\partial \boldsymbol{n}} = 0$ in the last step.
This is because we can factor out $\phi_i(\boldsymbol{x})^{\mu_i} (\boldsymbol{x})$ as follows
\begin{equation}
	v(\boldsymbol{x}) = \sum_{k=1, k\neq i}^{M} \prod_{j=1, j \neq k}^M \phi_j (\boldsymbol{x})^{\mu_j} = \phi_i(\boldsymbol{x})^{\mu_i} \sum_{k=1, k\neq i}^{M} \prod_{j=1, j \neq k, j\neq i}^M \phi_j (\boldsymbol{x})^{\mu_j},
\end{equation}
and that $\phi_i(\boldsymbol{x})^{\mu_i}= \frac{\partial \phi_i(\boldsymbol{x})^{\mu_i}}{\partial \boldsymbol{n}} = 0$.
Now we can further simplifyg~\eqref{eq:proof_prop_robin_condition_u_tilde_i} by substituting the ansatz for $\boldsymbol{u}_{\tilde{i}}$ given by~\eqref{local_solution_structure_system_ansatz_combination_of_basis_functions}, so that
\begin{equation}
	\frac{\partial \left( \boldsymbol{b}_{\tilde{i}}^{(\tilde{j})} (\boldsymbol{x}) \cdot \boldsymbol{u} (\boldsymbol{x}) \right)}{\partial \boldsymbol{n}}  
	= \frac{\partial \left( \boldsymbol{b}_{\tilde{i}}^{(\tilde{j})} (\boldsymbol{x}) \cdot \boldsymbol{u}_{\tilde{i}} (\boldsymbol{x}) \right)}{\partial \boldsymbol{n}}
	= \frac{\partial \left( \boldsymbol{b}_{\tilde{i}}^{(\tilde{j})} (\boldsymbol{x}) \cdot \sum_{j=1}^{n} \boldsymbol{b}_{\tilde{i}}^{(j)} (\boldsymbol{x}) u_{\tilde{i}}^{(j)} (\boldsymbol{x}) \right)}{\partial \boldsymbol{n}}.
\end{equation}
For further simplification, we use that $\boldsymbol{b}_{\tilde{i}}^{(\tilde{j})} (\boldsymbol{x}) \cdot \boldsymbol{b}_{\tilde{i}}^{(j)} (\boldsymbol{x}) = \delta_{\tilde{j} j}$ not only holds on $\Gamma_{\tilde{i}}$ but in an open neighbourhood of $\Gamma_{\tilde{i}}$.
Thus, we can say
\begin{equation}
	\frac{\partial \left( \boldsymbol{b}_{\tilde{i}}^{(\tilde{j})} (\boldsymbol{x}) \cdot \boldsymbol{u} (\boldsymbol{x}) \right)}{\partial \boldsymbol{n}} 
	= \frac{\partial \left( \sum_{j=1}^{n} \boldsymbol{b}_{\tilde{i}}^{(\tilde{j})} (\boldsymbol{x}) \cdot \boldsymbol{b}_{\tilde{i}}^{(j)} (\boldsymbol{x}) u_{\tilde{i}}^{(j)} (\boldsymbol{x}) \right)}{\partial \boldsymbol{n}}
	= \frac{\partial u_{\tilde{i}}^{(\tilde{j})} (\boldsymbol{x})}{\partial \boldsymbol{n}},
\end{equation}
which is what we wanted to show.

\subsubsection{Proof that the solution structure satisfies the boundary conditions}\label{sec:appendix_proof}
We want to show that the solution structure~\eqref{generic_solution_structure_system} satisfies the boundary conditions with local solution structures being chosen either according to the GLSS or the OP approach.
We begin by proving that the solution structure satisfies the Dirichlet boundary conditions.
Let $(i,j) \in \mathit{IJ}_D$ and $\boldsymbol{x} \in \Gamma_i$.
Using Proposition \ref{prop:bc-system} and substituting~\eqref{eq:ansatz_local_solution_structure_dirichlet_system_glss_op} for the local solution structure results in
\begin{equation}
	\boldsymbol{b}_i^{(j)} (\boldsymbol{x}) \cdot \boldsymbol{u} (\boldsymbol{x}) = u_i^{(j)} (\boldsymbol{x}) = \begin{cases}
		g_i^{(j)} (\boldsymbol{x}) + \bar{\phi}_i (\boldsymbol{x}) \tilde{\Psi}_i^{(j)} (\boldsymbol{x}), & \phi_{i}^{\mu_i} \text{ has a vanishing gradient}, \\
		g_i^{(j)} (\boldsymbol{x}), & \text{else.}
	\end{cases}
\end{equation}
Because of $\bar{\phi}_i (\boldsymbol{x}) = 0$ for $\boldsymbol{x} \in \Gamma_i$, it follows that 
\begin{equation}
	\boldsymbol{b}_i^{(j)} (\boldsymbol{x}) \cdot \boldsymbol{u} (\boldsymbol{x}) = g_i^{(j)} (\boldsymbol{x}),
\end{equation}
which is what we wanted to show.

To show that the Robin conditions are satisfied, let $(i,j) \in \mathit{IJ}_R$ and $\boldsymbol{x} \in \Gamma_i$.
By Proposition~\ref{prop:bc-system} we have
\begin{equation} \label{eq:prop_robin_in_proof_appendix}
	\frac{\partial \left( \boldsymbol{b}_i^{(j)} (\boldsymbol{x}) \cdot \boldsymbol{u} (\boldsymbol{x}) \right)}{\partial \boldsymbol{n}} 
	= \frac{\partial u_i^{(j)} (\boldsymbol{x})}{\partial \boldsymbol{n}}.
\end{equation}
First we consider the GLSS approach.
Using~\eqref{eq:ansatz_local_solution_structure_system_glss} for the local solution structure according yields
\begin{equation}
	\frac{\partial \left( \boldsymbol{b}_i^{(j)} (\boldsymbol{x}) \cdot \boldsymbol{u} (\boldsymbol{x}) \right)}{\partial \boldsymbol{n}} 
	= \frac{\partial \left( \Psi_i^{(j)} (\boldsymbol{x}) - \bar{\phi}_i (\boldsymbol{x}) \nabla \bar{\phi}_i (\boldsymbol{x}) \cdot \nabla \Psi_i^{(j)} (\boldsymbol{x}) \right) }{\partial \boldsymbol{n}}
	+ \frac{\partial \left( \bar{\phi}_i (\boldsymbol{x}) f_i^{(j)} (\boldsymbol{x}) \right) }{\partial \boldsymbol{n}}.
\end{equation}
The first term is zero because
\begin{multline}
	\frac{\partial \left( \Psi_i^{(j)} (\boldsymbol{x}) - \bar{\phi}_i (\boldsymbol{x}) \nabla \bar{\phi}_i (\boldsymbol{x}) \cdot \nabla \Psi_i^{(j)} (\boldsymbol{x}) \right) }{\partial \boldsymbol{n}}
	= \frac{\partial \Psi_i^{(j)} (\boldsymbol{x})}{\partial \boldsymbol{n}}
	- \underbrace{\frac{\partial \bar{\phi}_i (\boldsymbol{x})}{\partial \boldsymbol{n}}}_{=-1}
	\underbrace{\nabla \bar{\phi}_i (\boldsymbol{x}) \cdot \nabla \Psi_i^{(j)} (\boldsymbol{x})}_{=-\frac{\partial \Psi_i^{(j)} (\boldsymbol{x})}{\partial \boldsymbol{n}}} \\
	- \underbrace{\bar{\phi}_i (\boldsymbol{x})}_{=0}
	\frac{\partial \nabla \bar{\phi}_i (\boldsymbol{x}) \cdot \nabla \Psi_i^{(j)} (\boldsymbol{x})}{\partial \boldsymbol{n}}
	= 0.
\end{multline}
Thus, we get
\begin{equation}
	\begin{split}
	\frac{\partial \left( \boldsymbol{b}_i^{(j)} (\boldsymbol{x}) \cdot \boldsymbol{u} (\boldsymbol{x}) \right)}{\partial \boldsymbol{n}} 
	= \frac{\partial \left( \bar{\phi}_i (\boldsymbol{x}) f_i^{(j)} (\boldsymbol{x}) \right) }{\partial \boldsymbol{n}}
	= \underbrace{\frac{\partial \bar{\phi}_i (\boldsymbol{x})}{\partial \boldsymbol{n}}}_{=-1}
	f_i^{(j)} (\boldsymbol{x})
	+ \underbrace{\bar{\phi}_i (\boldsymbol{x})}_{=0} 
	\frac{\partial f_i^{(j)} (\boldsymbol{x})}{\partial \boldsymbol{n}}\\
	= -f_i^{(j)} (\boldsymbol{x}).
	\end{split}
\end{equation}
Substituting~\eqref{eq:ansatz_f_i_j} for  $f_i^{(j)}$ yields
\begin{multline} \label{eq:proof_GLSS_satisfies_Robin_condition}
	\frac{\partial \left( \boldsymbol{b}_i^{(j)} (\boldsymbol{x}) \cdot \boldsymbol{u} (\boldsymbol{x}) \right)}{\partial \boldsymbol{n}} 
	= - \boldsymbol{c}_i^{(j)} (\boldsymbol{x}) \cdot  
	\sum_{k=1, (i,k) \not\in \mathit{IJ}_D}^{n} \boldsymbol{b}_i^{(k)} (\boldsymbol{x}) \Psi_i^{(k)} (\boldsymbol{x})
	+ h_i^{(j)} (\boldsymbol{x}) \\
	= - \boldsymbol{c}_i^{(j)} (\boldsymbol{x}) \cdot  
	\sum_{k=1, (i,k) \not\in \mathit{IJ}_D}^{n} \boldsymbol{b}_i^{(k)} (\boldsymbol{x}) \Psi_i^{(k)} (\boldsymbol{x})
	\underbrace{- \boldsymbol{c}_i^{(j)} (\boldsymbol{x}) \cdot
		\sum_{k=1, (i,k) \in \mathit{IJ}_D}^{n} \boldsymbol{b}_i^{(k)} (\boldsymbol{x}) u_i^{(k)} (\boldsymbol{x})}_{=0 \text{ because } \boldsymbol{c}_i^{(j)} (\boldsymbol{x}) \in \Span(\{ \boldsymbol{b}_i^{(k)} (\boldsymbol{x}) | (i,k) \not \in \mathit{IJ}_D \})}\\
	+ h_i^{(j)} (\boldsymbol{x}) \\
	= -  \boldsymbol{c}_i^{(j)} (\boldsymbol{x}) \cdot  
	\underbrace{\sum_{k=1}^{n} \boldsymbol{b}_i^{(k)} (\boldsymbol{x}) u_i^{(k)} (\boldsymbol{x})}_{\boldsymbol{u}_i(\boldsymbol{x})=\boldsymbol{u}(\boldsymbol{x}) \text{ on } \Gamma_i}
	+h_i^{(j)} (\boldsymbol{x}) = -  \boldsymbol{c}_i^{(j)} (\boldsymbol{x}) \cdot  \boldsymbol{u} (\boldsymbol{x}) +h_i^{(j)} (\boldsymbol{x}),
\end{multline}
which is what we wanted to show.
Now, we want to show the same for the case that the local solution structure has been choosen according to OP approach.
We use $u_i^{(j)}$ from~\eqref{eq:prop_robin_in_proof_appendix} in~\eqref{eq:ansatz_local_solution_structure_op_system} to get
\begin{equation}
	\frac{\partial \left( \boldsymbol{b}_i^{(j)} (\boldsymbol{x}) \cdot \boldsymbol{u} (\boldsymbol{x}) \right)}{\partial \boldsymbol{n}} 
	= \frac{\partial \bar{\Psi}^{(j)} ( \mathcal{N} (\boldsymbol{x}; \bar{\phi}_i) ) }{\partial \boldsymbol{n}}
	+ \frac{\partial  \bar{\phi}_i (\boldsymbol{x}) f_i^{(j)} ( \mathcal{N} (\boldsymbol{x}; \bar{\phi}_i) )  }{\partial \boldsymbol{n}}.
\end{equation}
That the first term is zero can be seen by applying the chain rule and using that
\begin{equation}
	\frac{\partial \mathcal{N} (\boldsymbol{x}; \bar{\phi}_i)}{\partial \boldsymbol{n}}
	= \frac{\partial \boldsymbol{x} }{\partial \boldsymbol{n}}
	- \frac{\partial \bar{\phi}_i (\boldsymbol{x}) \nabla \bar{\phi}_i (\boldsymbol{x})}{\partial \boldsymbol{n}}
	= \boldsymbol{n} 
	- \underbrace{\frac{\partial \bar{\phi}_i (\boldsymbol{x})}{\partial \boldsymbol{n}}}_{=-1} 
	\underbrace{\nabla \bar{\phi}_i (\boldsymbol{x})}_{=-\boldsymbol{n}}
	- \underbrace{\bar{\phi}_i (\boldsymbol{x})}_{=0} 
	\frac{\partial \nabla \bar{\phi}_i (\boldsymbol{x})}{\partial \boldsymbol{n}}
	= \boldsymbol{0}.
\end{equation}
Similar as for GLSS, we conclude that
\begin{equation}
	\frac{\partial \left( \boldsymbol{b}_i^{(j)} (\boldsymbol{x}) \cdot \boldsymbol{u} (\boldsymbol{x}) \right)}{\partial \boldsymbol{n}} 
	= - f_i^{(j)} (\underbrace{\mathcal{N} (\boldsymbol{x}; \bar{\phi}_i)}_{=\boldsymbol{x}}),
\end{equation} 
In the same way wie did in~\eqref{eq:proof_GLSS_satisfies_Robin_condition}, we can conclude that the solution structure satisfies the boundary condition.
Note that because $\mathcal{N} (\boldsymbol{x}; \bar{\phi}_i) = \boldsymbol{x}$ holds for $\boldsymbol{x} \in \Gamma_i$, the orthogonal projection $\mathcal{N} (\boldsymbol{x}; \bar{\phi}_i)$ does not need to be applied to the function $f_i^{(j)}$ but only to the function $\bar{\Psi}^{(j)}$.

%% file: results.tex
We use the neural operator architecture from~\cite{li21}, which has the form
\begin{equation}
	\label{eq:pino}
	\mathcal{G}_{\boldsymbol{\theta}} = \mathcal{Q} \circ \sigma(\mathcal{W}_L + \mathcal{K}_L) \circ \cdots \circ \sigma(\mathcal{W}_1 + \mathcal{K}_1) \circ \mathcal{L},
\end{equation}
where $\mathcal{L}$ and $\mathcal{Q}$ are pointwise operators, $\mathcal{L}$ lifts the input function to a higher dimensional space and $\mathcal{Q}$ projects its input function back to lower dimensional space.
In between, each layer $\sigma(\mathcal{W}_l + \mathcal{K}_l)$ consists of a pointwise linear operator $\mathcal{W}_l$ and a kernel integral operator $\mathcal{K}_l$.
We use Fourier convolution operators for $\mathcal{K}_l$. 
Details are given by~\cite{li20b} and~\cite{li21}.
We use the  FNO  as implemented in \texttt{FNN2d}.\footnote{\url{https://github.com/neuraloperator/physics_informed/blob/Grad-ckpt/models/fourier2d.py}}
For all experiments, we use the architecture shown in~\eqref{eq:pino} with four Fourier convolution layers and GeLU activation functions. 
In every Fourier layer we have $20$ modes for each dimension.
Each of the four layers $\sigma (\mathcal{W}_l + \mathcal{K}_l)$ works in $64$-dimensional space, which requires that the layer $\mathcal{L}$ lifts the input to $64$-dimensional space.
Further, $\text{fcdim}=128$ is the dimension that $\mathcal{Q}$ lifts its inputs up to before projecting them back to the output dimension.

The computational domain for the FNO is required to be rectangular.
In experiments where this is not the case, we use a bounding box of the domain as computational domain for the FNO.
The rectangular domains are discretized on uniform grids.
For the Darcy flow problem the resolution of the uniform grid is $101 \times 101$ and $441 \times 83$ for the Navier-Stokes problem.
Derivatives in solution structures or losses are computed with finite differences on the uniform grids.

All four convolution layers together have $13,123,840$ trainable parameters in all experiments.
The lifting layer $\mathcal{L}$ has $256$ trainable parameters.
The only change in size is in the output dimension $I$ of $\mathcal{Q}$ and thus the dimensionality of the mapping $\mathcal{G}_{\boldsymbol{\theta}} (\boldsymbol{a}) : \Omega \rightarrow \mathbb{R}^{I}$.
Dimension $I$ depends on the used approach for the boundary conditions. 
For weakly enforced boundary conditions, where $\mathcal{G}_{\boldsymbol{\theta}} (\boldsymbol{a}) = \boldsymbol{u}$ is the solution, $I$ is equal to the dimension of the solution, that is $I=1$ for the scalar Darcy flow and $I=3$ for the Navier-Stokes equations where $\boldsymbol{u}$ contains two velocity components and the pressure.
For GLSS, OP and semi-weak boundary conditions, which rely on a solution structure, $I$ is equal to the number of functions $\Psi_i$ that need to be trained.
In our experiments, the $\mathcal{Q}$ has $8320 + 129 I$ trainable parameters, see Table~\ref{tab:training} for the overall number for the different boundary conditions.

Training uses the Adam optimizer~\cite{Kingma17} provided as part of~\cite{li21}'s FNO implementation.\footnote{\url{https://github.com/neuraloperator/physics_informed/blob/Grad-ckpt/train_utils/adam.py}}
All models are trained for a fixed number of total epochs.
In the training process, the learning rate is reduced by half at every milestone, see Table~\ref{tab:trainingparam}.
\begin{table}
	\centering
	\caption{Parameters used to train the FNO.}\label{tab:trainingparam}
	\begin{tabular}{llll}
		\toprule
		& \multicolumn{3}{l}{Darcy flow}
		\\
		\midrule
		& starting lr     & epochs per milestone     &  total epochs \\
		\cmidrule(r){2-4}
		PINN-like training     & 0.0025 & 200 & 1000 (each parameter) \\
		Operator training  & 0.005 & 100 & 1000 \\
		Finetuning   & 0.0025 & 200 & 1000 (each parameter) \\
		\midrule
		& \multicolumn{3}{l}{Navier-Stokes equations}
		\\
		\midrule
		& starting lr     & epochs per milestone     &  total epochs\\
		\cmidrule(r){2-4}
		PINN-like training     & 0.005 & 500 & 4000  \\
		\bottomrule
	\end{tabular}
\end{table}

All numerical experiments were conducted on a desktop computer with an Intel® Core™ i7-4790K CPU, 16 GB of RAM, and an NVIDIA GeForce GTX 960 GPU (4 GB).

\begin{table}[t]
		\centering
		\caption{Size, training and inference times of the FNO for the four different approaches to enforce boundary conditions for the Darcy flow and Navier-Stokes equations.}\label{tab:training}
	\begin{tabular}{ccccc}
	 & Trainable & Checkpoint size  & Training time  & Inference time \\
	 & parameters& (MByte) & (min) &  (sec)  \\ \hline
	 \multicolumn{5}{c}{Darcy flow} \\ \hline
	GLSS        & 13,132,932 & 105 & 182.35 & 0.0130 \\
	OP          & 13,132,674 & 105 & 181.68 & 0.0120 \\
	Semi-weak   & 13,132,545 & 105 & 180.79 & 0.0104 \\
	Weak        & 13,132,545 & 105 & 181.44 & 0.0101 \\ \hline
    \multicolumn{5}{c}{Navier-Stokes equations} \\ \hline  
	GLSS      & 13,133,706         & 105 & 9.79 & 0.0298 \\
	OP        & 13,133,190         & 105 & 9.76 & 0.0285 \\
	Semi-weak & 13,132,803         & 105 & 9.49 & 0.0248 \\
	Weak      & 13,132,803         & 105 & 9.62 & 0.0238 \\ \hline
	\end{tabular}	
\end{table}

Table~\ref{tab:training} shows different training-related parameters of the networks arising from the four approaches for the two benchmarks.
Because the size of the network is dominated by the size of the four convolution layers, the number of trainable parameters varies only very slightly.
There is no discernible impact on the size of the checkpoint files.
Training times are stable, with semi-weak boundary conditions training the fastest in both cases but the difference to the slowest GLSS is below $3\%$.
Inference times increase for GLSS and OP compared to weakly enforced boundary conditions.
We see the largest increase by about 29\% for GLSS for the Darcy flow.
Note that training for the Navier-Stokes equations is much faster because we train only solutions and no solution operator.
\subsection{Darcy Flow}\label{subsec:darcyy}
As scalar test problem, we consider the Darcy flow equation governing fluid flow in porous media~\cite{darcy}.
The governing PDE reads
\begin{equation}
	\label{eq:darcy_flow}
	- \nabla \cdot (a(x,y) \nabla u(x,y)) = f(x,y)
\end{equation}
Here, $a(x,y)$ is the inhomogeneous and anisotropic diffusivity and $f$ an inhomogeneous source.
For a pair of parameters $(\alpha, \beta) \in [1,10)^2$ we set the diffusivity in~\eqref{eq:darcy_flow} to
\begin{equation}
	\label{eq:darcy_flow_parameter}
	a(x,y) \coloneqq \sin(\alpha x) \sin(\beta y)
\end{equation}
and use the method of manufactured solutions to construct $f$ such that
\begin{equation}
	\label{eq:darcy_flow_u}
	u(x,y) = \sin(\alpha x) \cos(\beta y)
\end{equation}
becomes the sought solution. 
By inserting $a(x,y)$ from~\eqref{eq:darcy_flow_parameter} into~\eqref{eq:darcy_flow} we find
\begin{equation}
	f(x,y) = -\frac{1}{2}\sin(2 \beta y) (\alpha^2 \cos(2 \alpha x) + \beta^2 \cos(2 \alpha x) - \beta^2).
\end{equation}
Training data is generated by generating 500 pairs $(\alpha_i, \beta_i)$ from $[1,10)^2$ by uniform sampling.
The operator $\mathcal{G}$ is trained on 400 of those pairs.
For testing, it is evaluated on the remaining 100 pairs. 
The produced solution is compared against the analytic solution~\eqref{eq:darcy_flow_u} and the average $l_2$-error over those 100 pairs is computed.
Finetuning trains for the specific solution for each of the 100 remaining parameter pairs independently, computes the $l_2$ error against~\eqref{eq:darcy_flow_u} and averages over all 100 $l_2$-errors.
PINN-style training also trains the solutions for each of the 100 parameter pairs independently but does not start from the operator that was trained on the first 400 pairs but from the untrained network instead.
The first 400 pairs remain unused in this case.

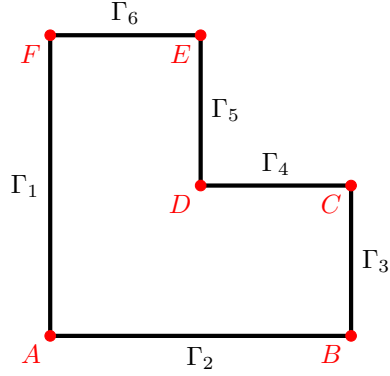
\begin{figure}
	\centering
	\begin{tikzpicture}
		\draw[ultra thick] (0,4)-- node[left] {$\Gamma_{1}$} (0,0);
		\draw[ultra thick] (0,0)-- node[below] {$\Gamma_{2}$} (4,0);
		\draw[ultra thick] (4,0)-- node[right] {$\Gamma_{3}$} (4,2);
		\draw[ultra thick] (2,2)-- node[above] {$\Gamma_{4}$} (4,2);
		\draw[ultra thick] (2,2)-- node[right] {$\Gamma_{5}$} (2,4);
		\draw[ultra thick] (0,4)-- node[above] {$\Gamma_{6}$} (2,4);
		
		\filldraw [red] (0,0) circle (2pt) node[below left] {$A$};
		\filldraw [red] (4,0) circle (2pt) node[below left] {$B$};
		\filldraw [red] (4,2) circle (2pt) node[below left] {$C$};
		\filldraw [red] (2,2) circle (2pt) node[below left] {$D$};
		\filldraw [red] (2,4) circle (2pt) node[below left] {$E$};
		\filldraw [red] (0,4) circle (2pt) node[below left] {$F$};
	\end{tikzpicture}
	\caption{Computational domain for Darcy flow. Dirichlet data is prescribed on $\Gamma_1, \Gamma_2$, Neumann boundary conditions are set for $\Gamma_3, \Gamma_4$, and Robin conditions on $\Gamma_5,\Gamma_6$.}
	\label{fig:L_shape}
\end{figure}

We consider the L-shaped domain shown in Figure~\ref{fig:L_shape} and derive the following boundary conditions from the analytic solution~\eqref{eq:darcy_flow_u}\begin{align}
	\forall (x,y) & \in \Gamma_1 : &
	u(x,y) & = g_1(x,y) = 0 \\
	\forall (x,y) & \in \Gamma_2 : &
	u(x,y) & = g_2(x,y) = \sin(\alpha x) \\
	\forall (x,y) & \in \Gamma_3 : &
	\frac{\partial u(x,y)}{\partial \boldsymbol{n}} & = h_3(x,y) = \alpha \cos(\alpha) \cos(\beta y) \\
	\forall (x,y) & \in \Gamma_4 : &
	\frac{\partial u(x,y)}{\partial \boldsymbol{n}} & = h_4(x,y) = -\beta \sin(\alpha x) \sin(0.5 \beta) \\
	\forall (x,y) & \in \Gamma_5 : &
	\frac{\partial u(x,y)}{\partial \boldsymbol{n}} + u(x,y) & = h_5(x,y) = ( \alpha \cos(0.5 \alpha) + \sin(0.5 \alpha) ) \cos(\beta y) \\
	\forall (x,y) & \in \Gamma_6 : &
	\frac{\partial u(x,y)}{\partial \boldsymbol{n}} + u(x,y) & = h_6(x,y) = (\cos(\beta) - \beta \sin(\beta)) \sin(\alpha x).
\end{align}
The resulting solution structure is
\begin{multline}
	u(x,y) = \\\sum_{i=1}^{6} w_i (x,y) u_i(x,y)
	+ \Psi (x,y) \phi_1 (x,y) \phi_2 (x,y) \phi_3 (x,y)^2 \phi_4 (x,y)^2 \phi_5 (x,y)^2 \phi_6 (x,y)^2.
\end{multline}
The local solution structures $u_i$ are choosen depending on whether GLSS or OP is used.

\paragraph{GLSS.} Here, the local solution structures are
\begin{align*}
	u_1 (x,y) & = g_1 (x,y), \\
	u_2 (x,y) & = g_2 (x,y), 
\end{align*}
for Dirichlet boundary conditions according to~\eqref{local_solution_structure_Dirichlet_scalar}, 
\begin{align*}
	u_3 (x,y) & = \Psi_3 (x,y) - \bar{\phi}_3 (x,y) ( \nabla \bar{\phi}_3 (x,y) \cdot \nabla \Psi_3 (x,y) + h_3 (x,y) ), \\
	u_4 (x,y) & = \Psi_4 (x,y) - \bar{\phi}_4 (x,y) ( \nabla \bar{\phi}_4 (x,y) \cdot \nabla \Psi_4 (x,y) + h_4 (x,y) ), 
\end{align*}
for Neumann boundary conditions and
\begin{align*}
	u_5 (x,y) & = \Psi_5 (x,y) (1+\bar{\phi}_5 (x,y)) - \bar{\phi}_5 (x,y) ( \nabla \bar{\phi}_5 (x,y) \cdot \nabla \Psi_5 (x,y) + h_5 (x,y) ), \\
	u_6 (x,y) & = \Psi_6 (x,y) (1+\bar{\phi}_6 (x,y)) - \bar{\phi}_6 (x,y) ( \nabla \bar{\phi}_6 (x,y) \cdot \nabla \Psi_6 (x,y) + h_6 (x,y) )
\end{align*}
for Robin boundary conditions according to~\eqref{local_solution_structure_Robin_scalar}.
The functions $\Psi_3$, $\Psi_4$, $\Psi_5$ and $\Psi_6$ are not unknowns because their corresponding boundary segements have intersection points with other segements.
That is why, we choose them as given in~\eqref{eq:Psi_i_GLSS_A_B} and introduce new unknowns $\Psi_C$, $\Psi_D$ and $\Psi_E$ that correspond to the intersection points.
With that, we set 
\begin{align}
	\Psi_3 (x,y) & = \frac{\phi_C (x,y)}{\phi_B (x,y) + \phi_C (x,y)} g_2(x,y)
	+ \frac{\phi_B (x,y)}{\phi_B (x,y) + \phi_C (x,y)} \Psi_C (x,y), \\
	\Psi_4 (x,y) & = \frac{\phi_D (x,y)}{\phi_C (x,y) + \phi_D (x,y)} \Psi_C(x,y)
	+ \frac{\phi_C (x,y)}{\phi_C (x,y) + \phi_D (x,y)} \Psi_D (x,y), \\
	\Psi_5 (x,y) & = \frac{\phi_E (x,y)}{\phi_D (x,y) + \phi_E (x,y)} \Psi_D(x,y)
	+ \frac{\phi_D (x,y)}{\phi_D (x,y) + \phi_E (x,y)} \Psi_E (x,y), \\
	\Psi_6 (x,y) & = \frac{\phi_F (x,y)}{\phi_E (x,y) + \phi_F (x,y)} \Psi_E(x,y)
	+ \frac{\phi_E (x,y)}{\phi_E (x,y) + \phi_F (x,y)} g_1 (x,y).
\end{align}
The unknown functions $\Psi$, $\Psi_C$, $\Psi_D$ and $\Psi_E$ are approximated by the PINO.

\paragraph{OP.} Here, the local solution structures are
\begin{align*}
	u_1 (x,y) & = g_1 (x,y), \\
	u_2 (x,y) & = g_2 (x,y), 
\end{align*}
for Dirichlet boundary conditions,
\begin{align*}  
	u_3 (x,y) & = \bar{\Psi} (\mathcal{N}(x,y;\bar{\phi}_3)) - \bar{\phi}_3 (x,y) h_3 (x,y), \\
	u_4 (x,y) & = \bar{\Psi} (\mathcal{N}(x,y;\bar{\phi}_4)) - \bar{\phi}_4 (x,y) h_4 (x,y), 
\end{align*}
for Neumann boundary conditions and  
\begin{align*}
	u_5 (x,y) & = \bar{\Psi} (\mathcal{N}(x,y;\bar{\phi}_5)) (1+\bar{\phi}_5 (x,y)) - \bar{\phi}_5 (x,y) h_5 (x,y),  \\
	u_6 (x,y) & = \bar{\Psi} (\mathcal{N}(x,y;\bar{\phi}_6)) (1+\bar{\phi}_6 (x,y)) - \bar{\phi}_6 (x,y) h_6 (x,y)
\end{align*}
for Robin boundary conditions according to~\eqref{local_solution_structure_robin_condition_op}.
We ensure that $\bar{\Psi}$ satisfies all Dirichlet conditions by following~\eqref{eq:scalar_op_psi_i_dirichlet_fitting} and setting 
\begin{multline}
	\bar{\Psi} (x,y) = \frac{\phi_2 (x,y)}{\phi_1 (x,y) + \phi_2 (x,y)} g_1 (x,y) + \frac{\phi_1 (x,y)}{\phi_1 (x,y) + \phi_2 (x,y)} g_2 (x,y) \\+ \phi_1 (x,y) \phi_2 (x,y) \tilde{\Psi} (x,y),
\end{multline}
where $\Psi$ and $\tilde{\Psi}$ need to be trained.

Figure~\ref{fig:darcy_training} shows training loss (upper) and validation error (lower).
The left column shows the loss and error curve for the PINN-like training.
Both these curves are the average loss and $l^2$-error over the 100 parameters that the PINO was trained on individually in PINN-style.
The middle column shows the loss and error curve for PINO trained on 400 parameters.
The right column shows the loss and error curve for finetuning,
Dotted lines indicate the loss and error value at the very beginning of the finetuning.

All cases train reasonably well, reducing the loss function by at least one order of magnitude (OP for PINN-like training) and two or more orders in most cases.
Losses are not indicative of achieved validation errors.
For the PINN-style training, OP and GLSS are more accurate than weak or semi-weak boundary conditions.
The same holds true for finetuning, where OP is slightly more accurate than GLSS.
For operator training, OP is more accurate than GLSS which performs on par with semi-weak and better than weak boundary conditions.
In summary, for the Darcy flow, even though losses do not necessarily decay faster, OP and GLSS in almost all cases produces more accurate solutions than weak or semi-weak boundary conditions.
\begin{figure}[t]
  \centering
  \includegraphics[width=0.9\textwidth]{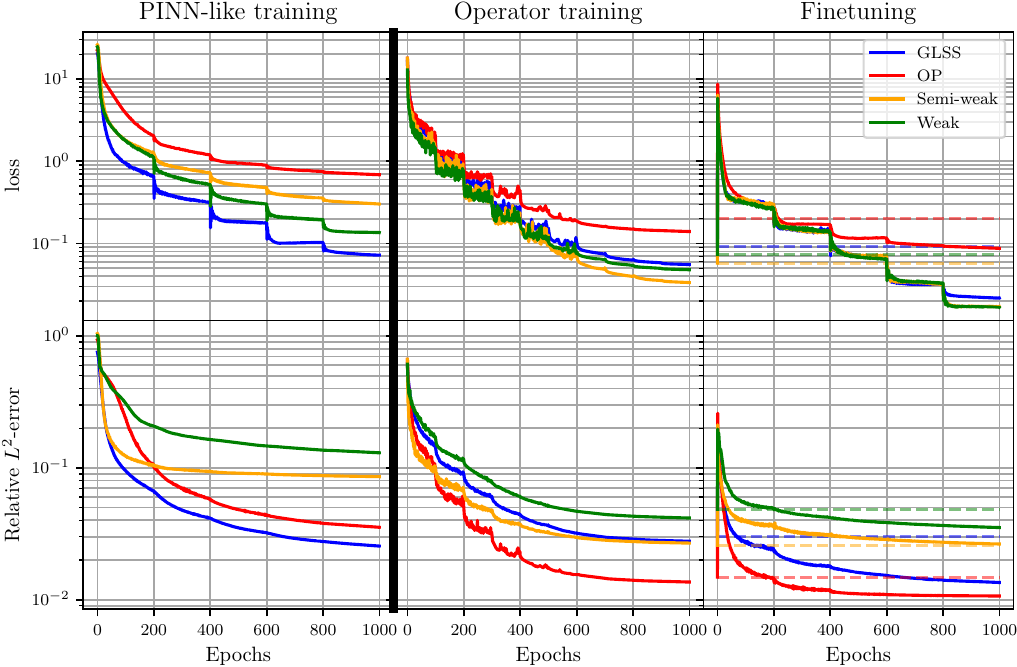}
  \caption{Training progress and errors on the validation set for different ways to enforce boundary conditions for the Darcy flow.}
  \label{fig:darcy_training}
\end{figure}
Table~\ref{operator_training_highest_lowest_median} shows the average $l_2$-error plus standard deviation (left column), best case $l_2$-error (middle column) and worst case $l_2$-error (right column).
For operator training and finetuning, OP is the most accurate approach whilst GLSS is the most accurate for PINN-like training.
For best case errors, shown in the middle column, there is no clear benefit from the two new approaches 
However, there are substantial gains in accuracy from OP and GLSS for the worst case in PINN-like training and finetuning and from OP in operator training.
Plots of the median, best- and worst-case solutions can be found in Appendix~\ref{sec:best-worst-darcy}.
\begin{table}[th]
	  \caption{$l_2$-errors of the predicted $u$ against the analytical solution for the four different approaches to enforce boundary conditions for the Darcy flow problem.}
	\label{operator_training_highest_lowest_median}  
  \centering
  \begin{tabular}{|l|lll|lll|lll|}
    \toprule
     & \multicolumn{3}{|l|}{Operator training} & \multicolumn{3}{|l|}{Finetuning} & \multicolumn{3}{|l|}{PINN-like training} \\
     \midrule
                       & Average     & Best  &  Worst  & Average  & Best   &  Worst  &  Average  & Best  &  Worst \\
    \midrule
    GLSS          &0.03$\pm$0.04&0.004& 0.27  & 0.01$\pm$0.01 & 0.003 & 0.06 &  0.02$\pm$0.02 & 0.005 & 0.08  \\
    OP               &0.02$\pm$0.01&0.003& 0.06  & 0.01$\pm$0.01 & 0.002  & 0.04  &  0.04$\pm$0.02 & 0.006 & 0.11  \\
    S-Weak       &0.03$\pm$0.03&0.002& 0.17  &  0.03$\pm$0.02 & 0.003 & 0.10 &  0.09$\pm$0.03 & 0.021 & 0.17 \\
    Weak           &0.05$\pm$0.05&0.008& 0.28 &  0.04$\pm$0.05 & 0.003 & 0.26  & 0.13$\pm$0.14 & 0.008 & 0.64 \\
    \bottomrule
    \end{tabular}
\end{table}
\subsection{Navier-Stokes equations}\label{subsec:nse}
We use the standard benchmark by~\cite{Turek96}, simulating 2D stationary flow through a channel and across a cylinder. 
The domain is sketched in Figure~\ref{fig:nse_domain}.
The benchmark solves the stationary, incompressible Navier-Stokes equations 
\begin{align}
	- \nu \nabla^2 \boldsymbol{u} (x,y) + (\boldsymbol{u} \cdot \nabla) \boldsymbol{u} (x,y) + \nabla p (x,y) & = 0, & (x,y) & \in \Omega, \\
	\nabla \cdot \boldsymbol{u} (x,y) & = 0, & (x,y) & \in \Omega
\end{align}
with  inflow boundary condition $\boldsymbol{u} (x,y) = \boldsymbol{u}_0 (x,y)$ on $\Gamma_{\text{in}} = \Gamma_1$, no-slip boundary condition  $\boldsymbol{u} (x,y) = 0$ on $\Gamma_{\text{wall}} = \Gamma_2 \cup \Gamma_4 \cup \Gamma_S$ and outflow boundary condition  $\nu \frac{\partial \boldsymbol{u} (x,y)}{\partial \boldsymbol{n}} - \boldsymbol{n} (x,y) p (x,y)  = 0$ on $\Gamma_{\text{out}} = \Gamma_3$.
We set $\boldsymbol{u}_0 (x,y) = [u_0 (y), 0]^T$ and $\nu = 0.001$.
Only PINN-like training is investigated in this case.
Note that we rescale the pressure by $\tilde{p} \coloneqq \frac{p}{\sqrt{\nu}}$ to reduce the ratio between velocity derivative and pressure in the outflow condition, which provided benefits to the accuracy of the exact boundary conditions.
Since no analytical solution is available, we compute a reference solution numerically using FEniCSx~\cite{Baratta23,Scroggs22a,Scroggs22b,Alnaes14}.\footnote{We use a modified version of the code provided in the corresponding tutorial: \url{https://jsdokken.com/dolfinx-tutorial/chapter2/ns_code2.html}.}

We consider the boundary conditions in the form of \eqref{eq:Dirichlet_boundary_condition_system_case} and \eqref{eq:Robin_boundary_condition_system_case}.
For all $\Gamma_i$, we pose the boundary conditions with respect to the same basis vectors $b_i^{(1)} (\boldsymbol{x}) = (1,0,0)^T$, $b_i^{(2)} (\boldsymbol{x}) = (0,1,0)^T$ and $b_i^{(3)} (\boldsymbol{x}) = (0,0,1)^T$.
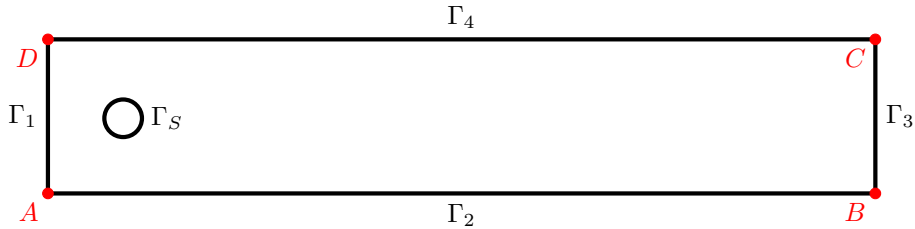
\begin{figure}[th]
	\centering
	\begin{tikzpicture}
		\draw[ultra thick] (0,0)-- node[below] {$\Gamma_{2}$} (11,0);
		\draw[ultra thick] (11,2.05)-- node[right] {$\Gamma_{3}$} (11,0);
		\draw[ultra thick] (11,2.05)-- node[above] {$\Gamma_{4}$} (0,2.05);
		\draw[ultra thick] (0,0)-- node[left] {$\Gamma_{1}$} (0,2.05);
		\draw[ultra thick] (1,1) circle (0.25);
		\draw (1.25,1) circle (0) node[right] {$\Gamma_{S}$};
		
		\filldraw [red] (0,0) circle (2pt) node[below left] {$A$};
		\filldraw [red] (11,0) circle (2pt) node[below left] {$B$};
		\filldraw [red] (11,2.05) circle (2pt) node[below left] {$C$};
		\filldraw [red] (0,2.05) circle (2pt) node[below left] {$D$};
	\end{tikzpicture}
	\caption{Computational domain for the Navier-Stokes test case.}
	\label{fig:nse_domain}
\end{figure}

\paragraph{GLSS.}
The solution structure generated by Algorithms~\ref{alg:intersecting_boundary_segments_system_part_1} and~\ref{alg:intersecting_boundary_segments_system_part_2} are
\begin{multline}
	\begin{pmatrix}
		u(x,y) \\ v(x,y) \\ \tilde{p}(x,y)
	\end{pmatrix}
	= w_1(x,y) \begin{pmatrix}
		u_0 (y) \\ 0 \\ \Psi_1^{\tilde{p}} (x,y)
	\end{pmatrix}
	+ w_2(x,y) \begin{pmatrix}
		0 \\ 0 \\ \Psi_2^{\tilde{p}} (x,y)
	\end{pmatrix} \\
	+ w_3(x,y) \begin{pmatrix}
		\Psi_3^{u} (x,y) - \bar{\phi}_3 (x,y) \nabla \bar{\phi}_3 (x,y) \cdot \nabla \Psi_3^{u}(x,y) + \bar{\phi}_3 (x,y) \Psi_3^{\tilde{p}} (x,y) / \sqrt{\nu} \\ 
		\Psi_3^{v} (x,y) - \bar{\phi}_3 (x,y) \nabla \bar{\phi}_3 (x,y) \cdot \nabla \Psi_3^{v}(x,y) \\ 
		\Psi_3^{\tilde{p}} (x,y)
	\end{pmatrix} \\
	+ w_4(x,y) \begin{pmatrix}
		0 \\ 0 \\ \Psi_4^{\tilde{p}} (x,y)
	\end{pmatrix} 
	+ w_S(x,y) \begin{pmatrix}
		0 \\ 0 \\ \Psi_S^{\tilde{p}} (x,y)
	\end{pmatrix}  \\
	+ \phi_1(x,y) \phi_2(x,y) \phi_3(x,y)^2 \phi_4(x,y) \phi_S(x,y)
	\begin{pmatrix}
		\Psi^{u}(x,y) \\
		\Psi^{v}(x,y) \\
		\Psi^{\tilde{p}}(x,y)
	\end{pmatrix},
\end{multline}
where
\begin{align}
	\Psi_3^{u} (x,y) & =
	\phi_B (x,y) \phi_C (x,y) \bar{\Psi}_3^{u} (x,y), \\
	\Psi_3^{v} (x,y) & =
	\phi_B (x,y) \phi_C (x,y) \bar{\Psi}_3^{v} (x,y), \\
	\Psi_1^{\tilde{p}} (x,y) & = 
	\frac{\phi_D(x,y)}{\phi_A(x,y) + \phi_D(x,y)} \Psi_A^{\tilde{p}}(x,y) + 
	\frac{\phi_A(x,y)}{\phi_A(x,y) + \phi_D(x,y)} \Psi_D^{\tilde{p}}(x,y),\\
	\Psi_2^{\tilde{p}} (x,y) & = 
	\frac{\phi_B(x,y)}{\phi_A(x,y) + \phi_B(x,y)} \Psi_A^{\tilde{p}}(x,y) + 
	\frac{\phi_A(x,y)}{\phi_A(x,y) + \phi_B(x,y)} \Psi_B^{\tilde{p}}(x,y),\\
	\Psi_3^{\tilde{p}} (x,y) & = 
	\frac{\phi_C(x,y)}{\phi_B(x,y) + \phi_C(x,y)} \Psi_B^{\tilde{p}}(x,y) + 
	\frac{\phi_B(x,y)}{\phi_B(x,y) + \phi_C(x,y)} \Psi_C^{\tilde{p}}(x,y),\\
	\Psi_4^{\tilde{p}} (x,y) & = 
	\frac{\phi_D(x,y)}{\phi_C(x,y) + \phi_D(x,y)} \Psi_C^{\tilde{p}}(x,y) + 
	\frac{\phi_C(x,y)}{\phi_C(x,y) + \phi_D(x,y)} \Psi_D^{\tilde{p}}(x,y)
\end{align}
The unknown functions to be learned are $\Psi^u$, $\Psi^v$, $\Psi^{\tilde{p}}$, $\Psi_S^{\tilde{p}}$, $\bar{\Psi}_3^{u}$, $\bar{\Psi}_3^{v}$, $\Psi_A^{\tilde{p}}$, $\Psi_B^{\tilde{p}}$, $\Psi_C^{\tilde{p}}$ and $\Psi_D^{\tilde{p}}$.

\paragraph{OP.}
According to~\eqref{eq:global_solution_structure_op_system_simplified} and~\eqref{eq:ansatz_local_solution_structure_op_system}, the solution structure for the variables $u$ and $v$ is
\begin{multline}
	\begin{pmatrix}
		u(x,y) \\ v(x,y) 
	\end{pmatrix}
	= w_1(x,y) \begin{pmatrix}
		u_0 (y) \\ 0 
	\end{pmatrix}
	+ w_2(x,y) \begin{pmatrix}
		0 \\ 0 
	\end{pmatrix} \\
	+ w_3(x,y) \begin{pmatrix}
		\bar{\Psi}^{u} (2.2,y) - \bar{\phi}_3 (x,y) \bar{\Psi}^{\tilde{p}} (x,y) / \sqrt{\nu} \\ 
		\bar{\Psi}^{v} (2.2,y)  
	\end{pmatrix} \\
	+ w_4(x,y) \begin{pmatrix}
		0 \\ 0 
	\end{pmatrix} 
	+ w_S(x,y) \begin{pmatrix}
		0 \\ 0
	\end{pmatrix}  \\
	+ \phi_1(x,y) \phi_2(x,y) \phi_3(x,y)^2 \phi_4(x,y) \phi_S(x,y)
	\begin{pmatrix}
		\Psi^{u}(x,y) \\
		\Psi^{v}(x,y) 
	\end{pmatrix},
\end{multline}
and for $\tilde{p}$ it would be
\begin{equation}
	\tilde{p} (x,y) = \bar{\Psi}^{\tilde{p}} (x,y) + \phi_1(x,y) \phi_2(x,y) \phi_3(x,y)^2 \phi_4(x,y) \phi_S(x,y) \Psi^{\tilde{p}} (x,y).
\end{equation}
However, we will use a simplified version of the solution structure for $\tilde{p}$ that we define by
\begin{equation}
	\tilde{p} (x,y) = \bar{\Psi}^{\tilde{p}} (x,y) + \phi_3 (x,y) \Psi^{\tilde{p}} (x,y).
\end{equation}
Without doing this, PINO has problems learning the solution properly.
The functions $\bar{\Psi}^{u}$, $\bar{\Psi}^{v}$ and $\bar{\Psi}^{\tilde{p}}$ are defined in a way such that they satisfy their corresponding Dirichlet conditions according to~\eqref{eq:system_op_bar_psi_j_tilde_psi_j}, i.e.
\begin{align}
	\begin{split}
		\bar{\Psi}^{u} (x,y) & = \frac{\phi_S(x,y)}{ \phi_1 (x,y) \phi_2 (x,y) \phi_4 (x,y) + \phi_S(x,y)} u_0(y) \\
		& \quad + \phi_1 (x,y) \phi_2 (x,y) \phi_4 (x,y) \phi_S (x,y) \tilde{\Psi}^{u} (x,y),
	\end{split} \\
	\bar{\Psi}^{v} (x,y) & = \phi_1 (x,y) \phi_2 (x,y) \phi_4 (x,y) \phi_S (x,y) \tilde{\Psi}^{u} (x,y), \\
	\bar{\Psi}^{\tilde{p}} (x,y) & = \tilde{\Psi}^{\tilde{p}} (x,y).
\end{align}
The unknown functions that need to be approximated are $\Psi^{u}$, $\Psi^{v}$, $\Psi^{\tilde{p}}$, $\tilde{\Psi}^{u}$, $\tilde{\Psi}^{v}$ and $\tilde{\Psi}^{\tilde{p}}$.

Figure~\ref{fig:nse_losses} shows the training losses in the upper figure and the errors in the velocity components $u$ and $v$ and the pressure $p$ against the numerically computed reference solution.
After 4000 epochs, losses for GLSS, OP and weak boundary conditions are similar but the loss for semi-weak remains somewhat higher.
In terms of errors, we again see a clear benefit in terms of accuracy from GLSS and OP as they outperform weak and semi-weak boundary conditions in all three solution components.
\begin{figure}[t]
  \centering
  \includegraphics[width=0.9\textwidth]{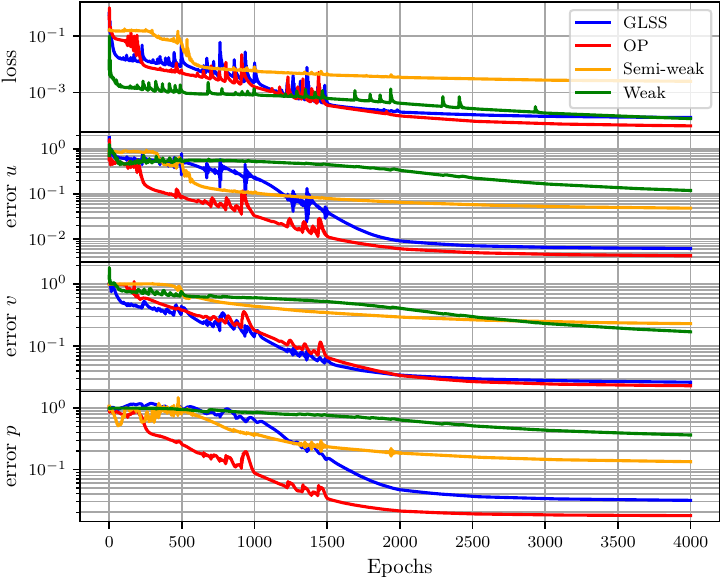}  
  \caption{Training loss (upper) and validation error in the $u$ (upper middle) and $v$ (lower middle) velocity component and pressure $p$ (lower) for the Navier-Stokes equations.}
  \label{fig:nse_losses}
\end{figure}

To further assess accuracy we consider three practically relevant diagnostic quantities:  pressure difference,  drag coefficient and lift coefficient, see~\cite{Turek96} for their definition.
Table~\ref{tab:nse_diagnostic} shows the values computed from the PINO using the four different ways to enforce boundary conditions and, in brackets, the relative error against the reference values by~\cite{nabh98}.
We again see a noticeable increase in accuracy from GLSS and OP over weak or semi-weak boundary conditions.
Pressure difference and drag coefficient are predicted with high accuracy.
While relative errors for the lift coefficient are large, they are still orders of magnitude smaller than for the weak or semi-weak approach.
\begin{table}[h]
		\centering
		\caption{Physically important parameters computed from the Navier-Stokes solution. The reference values are provided by~\cite{nabh98} with 9 digit accuracy and we rounded them to 4 digits. The relative error against those reference values is shown in brackets.}  	
	\label{tab:nse_diagnostic}
    \begin{tabular}{llll}
 	\toprule
    & Pressure difference     & Drag coefficient     & Lift coefficient \\ \midrule
    GLSS      & 0.1150 (\phantom{0}2.1\%) & 5.5336 (\phantom{0}0.8\%) & -0.0058 (\phantom{0}155\%) \\
 	OP        & 0.1145 (\phantom{0}2.6\%) & 5.5366 (\phantom{0}0.8\%) & \phantom{-}0.0024 (\phantom{00}77\%)\\
 	Semi-weak & 0.0678 (42.3\%)           & 3.8221 (31.5\%)           & -0.3759 (3646\%)   \\
 	Weak      & 0.0902 (23.2\%).          & 4.6633 (16.4\%)           & \phantom{-}0.3849 (3531\%) \\ \midrule
    Reference values & 0.1175 & 5.5795 & \phantom{-}0.0106 \\
  	\bottomrule
  	\end{tabular}
\end{table}

%% file: conclusion.tex
We have introduced solution structures that enforces boundary conditions for training physics-informed neural operators, both for scalar PDEs as well as systems. By relaxing assumptions, we show the construction of $C^1$ continuous ansatz functions even on boundaries that are only $C^0$ globally. Numerical examples demonstrate the improved performance of the two new approaches, GLSS and OP. The two new approaches come with some limitations and drawbacks. 
First, they increase complexity of implementation compared to weakly enforced boundary conditions, in particular if the number of $C^1$-segments that form the boundary is high.
Second, while the size of the network as well as training times change only marginally, we do see an increase of inference times of up to 30\%. This, however, can be improved by an optimized implementation. Third, the OP approach only works for boundaries that can be decomposed into segments where each of these segments lies in a hyperplane.
Finally, the approaches are only tested for $\Omega \subset \mathbb{R}^2$, although generalization to 3D should be possible.

%% file: appendix.tex
\section{Best- and worst-case solutions for the Darcy flow problem}\label{sec:best-worst-darcy}
Figures~\ref{fig:darcy-sol1}--\ref{fig:darcy-sol4} show the solutions produced by the different boundary conditions for the parameter pairs that produced the best- and worst-case errors in Table~\ref{operator_training_highest_lowest_median}.
In addition, the middles figures show the solution that provides the median error for the 100 evaluation pairs.
\begin{figure}[p!]
	\centering
	\includegraphics[width=0.9\textwidth]{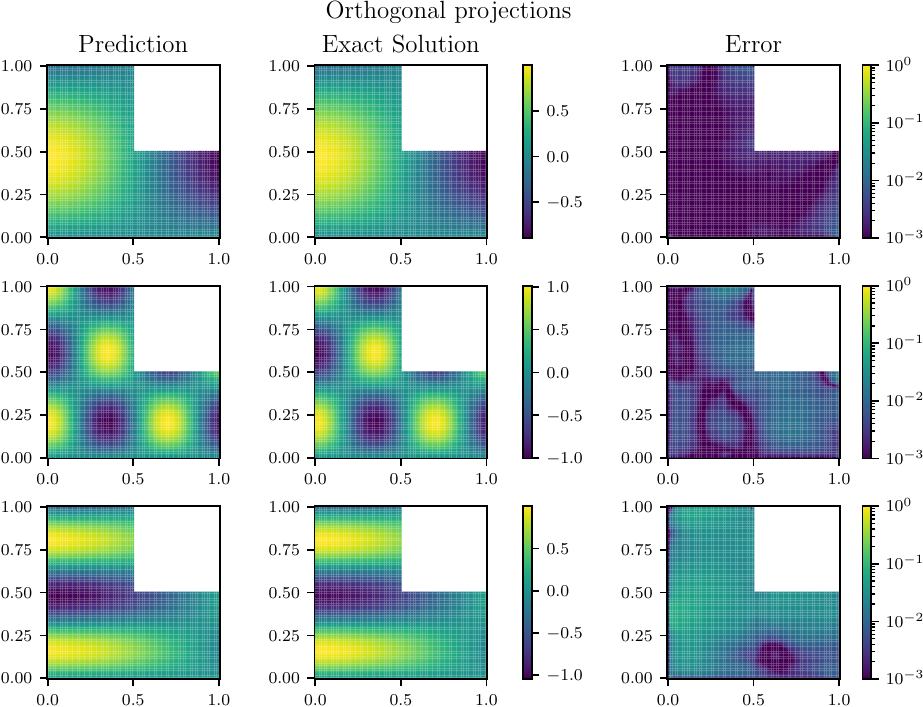}
	\caption{Best-case (upper), median (middle) and worst-case (lower) solution for the Darcy flow with OP boundary conditions.}
	\label{fig:darcy-sol1}
\end{figure}
\begin{figure}[p!]
	\centering
	\includegraphics[width=0.9\textwidth]{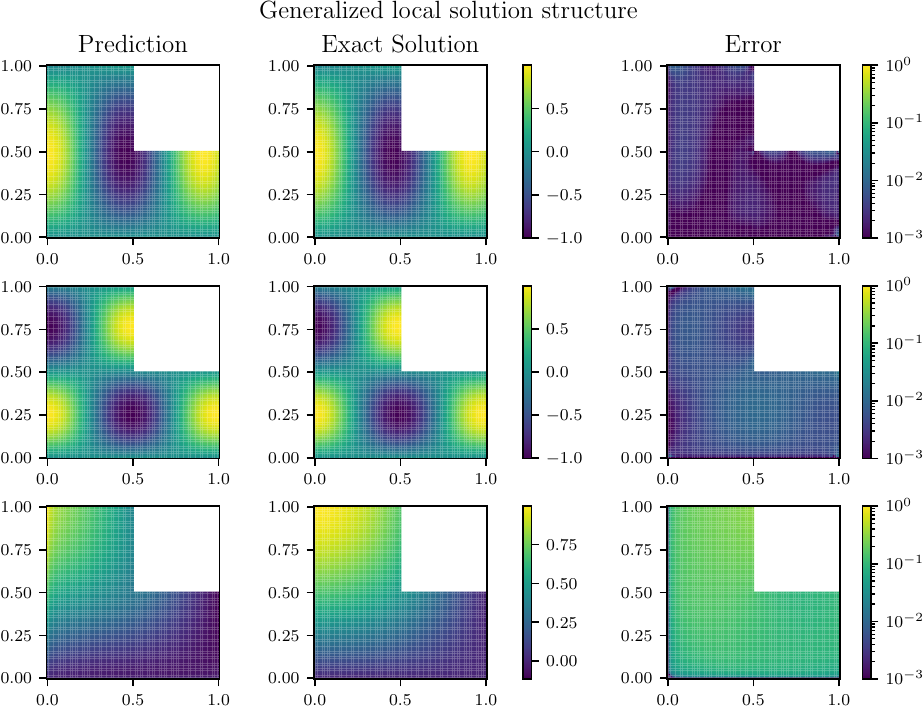}
	\caption{Best-case (upper), median (middle) and worst-case (lower) solution for the Darcy flow with GLSS boundary conditions.}
		\label{fig:darcy-sol2}
\end{figure}
\begin{figure}[p!]
	\centering
	\includegraphics[width=0.9\textwidth]{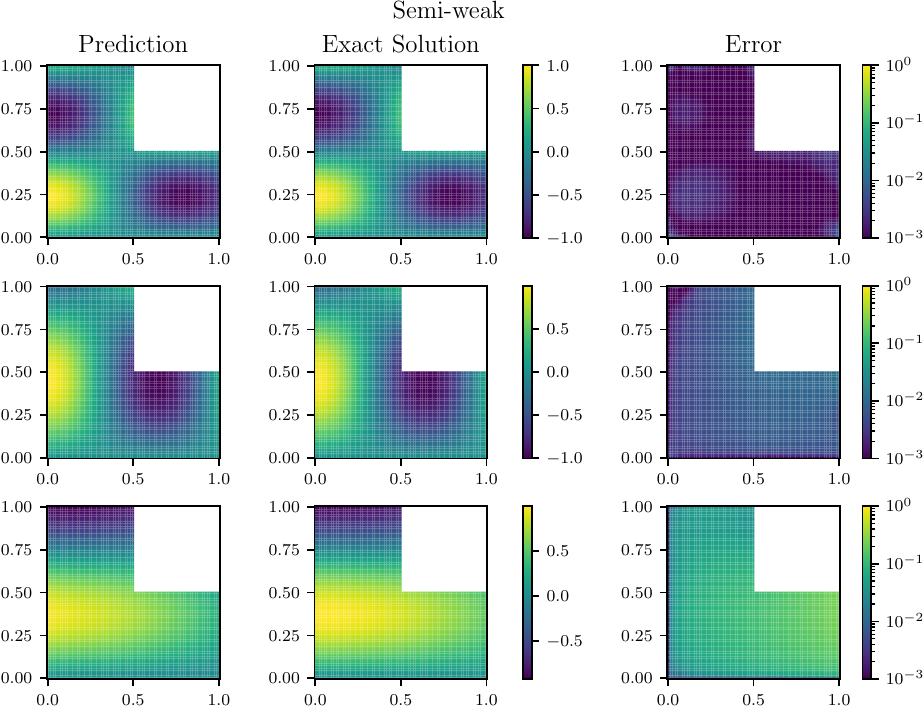}
	\caption{Best-case (upper), median (middle) and worst-case (lower) solution for the Darcy flow with semi-weak boundary conditions.}
		\label{fig:darcy-sol3}
\end{figure}
\begin{figure}[p!]
	\centering
	\includegraphics[width=0.9\textwidth]{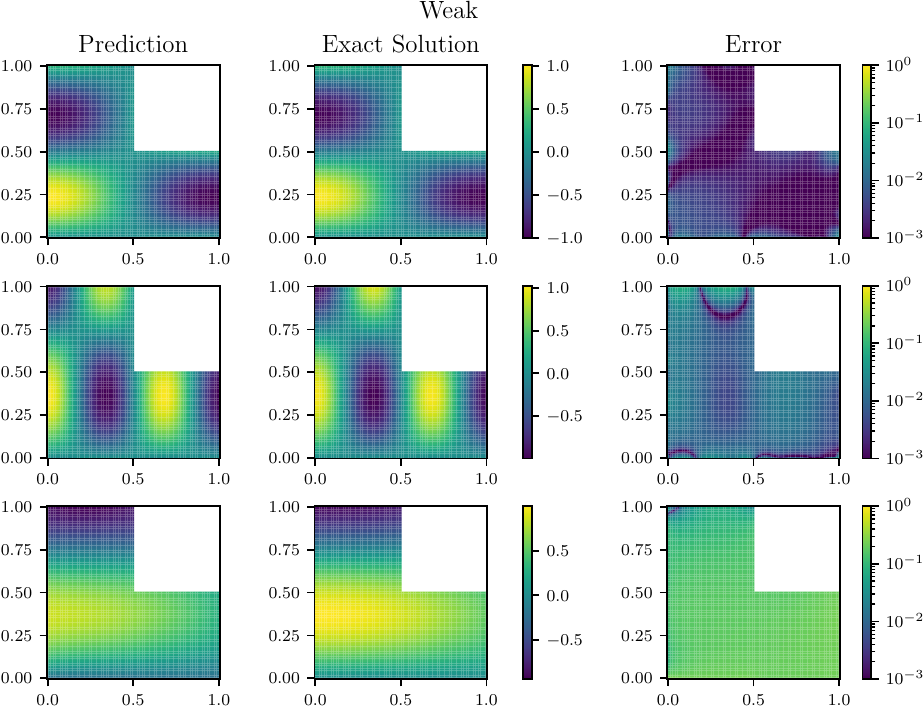}
	\caption{Best-case (upper), median (middle) and worst-case (lower) solution for the Darcy flow with weak boundary conditions.}
		\label{fig:darcy-sol4}
\end{figure}

\section{Solutions for the Navier-Stokes equations}
Figure~\ref{fig:nse-sol} shows the solutions to the Navier-Stokes benchmark problem produced by the PINN-style training of the PINO using the four different ways to prescribe boundary conditions as well as the numerical reference at the top.
While all solutions agree qualitatively, for weak and semi-weak boundaries, the peak in pressure right in front of the obstacle is too weak and the velocity is underestimated.
\begin{figure}[th]
  \centering
  \includegraphics[width=\textwidth]{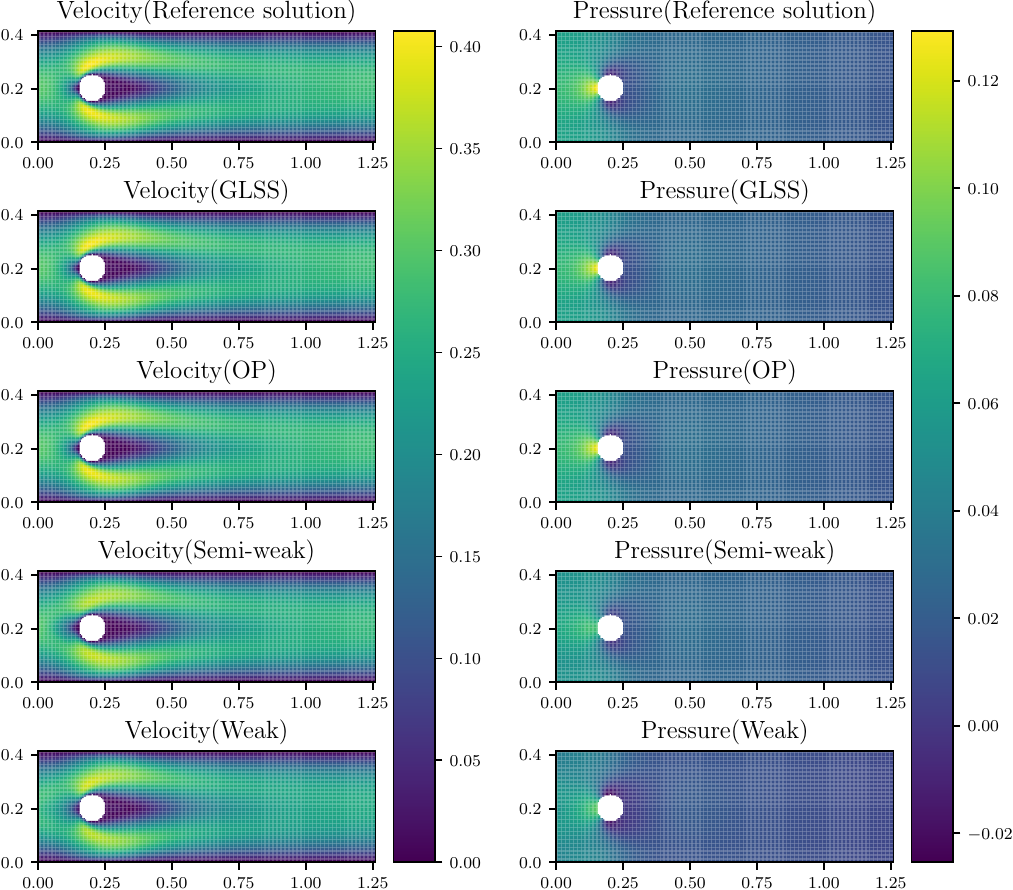}  
  \caption{Velocity (left) and pressure (right) produced by the FNO using the four different ways to prescribe boundary conditions. The uppermost figure shows the numerical reference solution.}
  \label{fig:nse-sol}
\end{figure}